\documentclass[journal,twoside,web]{ieeecolor}

\usepackage{generic}
\usepackage{cite}
\usepackage{amsmath,amssymb,amsfonts}
\usepackage[scr = dutchcal]{mathalpha}
\usepackage{graphicx}
\usepackage{textcomp}
\usepackage{stmaryrd}
\usepackage{tabularx}
\usepackage{multirow}
\usepackage{float}
\usepackage{mathtools}
\usepackage{comment}

\usepackage{color}

\newcommand{\mbf}[1]{\mathbf{ #1}}
\newcommand{\tbf}[1]{\textbf{#1}}

\newcommand{\mcl}[1]{\mathcal{#1}}
\newcommand{\mscr}[1]{\mathscr{#1}}

\newcommand{\R}{\mathbb{R}}
\newcommand{\N}{\mathbb{N}}

\newcommand{\ip}[2]{\left\langle #1, #2 \right\rangle}

\newcommand{\bmat}[1]{\begin{bmatrix}#1\end{bmatrix}}

\newcommand{\smallbmat}[1]{\left[\scriptsize\begin{smallmatrix}
		#1\end{smallmatrix} \right]}

\newcommand{\srbmat}[1]{\small\left[\!\!\!\begin{array}{r}#1\end{array}\!\!\!\right]}
\newcommand{\slbmat}[1]{\small\left[\!\!\!\begin{array}{l}#1\end{array}\!\!\!\right]}

\newtheorem{thm}{Theorem}
\newtheorem{defn}[thm]{Definition}
\newtheorem{lem}[thm]{Lemma}
\newtheorem{prop}[thm]{Proposition}

\let\bl\bigl
\let\bbl\Bigl
\let\bbbl\biggl

\let\br\bigr
\let\bbr\Bigr
\let\bbbr\biggr

\floatstyle{ruled}
\newfloat{block}{thbp}{lop}
\floatname{block}{Block}

\newcommand\Scale[2]{\scalebox{#1}{\mbox{\ensuremath{\displaystyle #2}}}}
\newcommand\Resize[2]{\resizebox{#1}{!}{\mbox{\ensuremath{\displaystyle #2}}}}

\title{\LARGE \bf
	A PIE Representation of Scalar Quadratic PDEs and Global Stability Analysis Using SDP
}

\author{Declan Jagt, Peter Seiler, Matthew Peet %
\thanks{\tbf{Acknowledgement:} This work was supported by National Science Foundation grants CMMI-1935453 and CMMI-1931270.} %
}

\begin{document}

	\maketitle
	\thispagestyle{empty}
	\pagestyle{plain}

\begin{abstract}		
	It has recently been shown that the evolution of a linear Partial Differential Equation (PDE) can be more conveniently represented in terms of the evolution of a higher spatial derivative of the state. This higher spatial derivative (termed the `fundamental state') lies in $L_2$ - requiring no auxiliary boundary conditions or continuity constraints. Such a representation (termed a Partial Integral Equation or PIE) is then defined in terms of an algebra of bounded integral operators (termed Partial Integral (PI) operators) and is constructed by identifying a unitary map from the fundamental state to the state of the original PDE. Unfortunately, when the PDE is nonlinear, the dynamics of the associated fundamental state are no longer parameterized in terms of PI operators. However, in this paper, we show that such dynamics can be compactly represented using a new tensor algebra of partial integral operators acting on the tensor product of the fundamental state. We further show that this tensor product of the fundamental state forms a natural distributed equivalent of the monomial basis used in representation of polynomials on a finite-dimensional space. This new representation is then used to provide a simple SDP-based Lyapunov test of stability of quadratic PDEs. The test is applied to three illustrative examples of quadratic PDEs.
\end{abstract}


\section{INTRODUCTION}

In this paper, we consider the problem of representation and stability analysis of quadratic Partial Differential Equations (PDEs). Quadratic PDEs are frequently used to model physical processes, including fluid dynamics (e.g. Navier-Stokes), population growth (e.g. Fisher's equation), and wave propagation (e.g. Korteweg-de Vries). However, certain aspects of the standard PDE representation of such nonlinear systems present difficulties when applied to the problems of analysis and simulation.
For example, consider a modified version of Burger's equation with Dirichlet Boundary Conditions (BCs),
\begin{align}\label{eq:Burgers_intro}
	u_{t}(t,s)&=u_{ss}(t,s)+ru(t,s) -u(t,s)u_{s}(t,s),	&	s&\in[0,1],	\notag\\
	u(t,0)&=0,\qquad u(t,1)=0,
\end{align}
wherein we added a reaction term $ru(t,s)$ to the dynamics.
To verify stability of this system, we can use the candidate Lyapunov Functional (LF) $V(u)=\frac{1}{2}\|u\|_{L_2}^2\geq 0$. Using integration by parts, and invoking the BCs, it can be shown that the derivative of this LF along solutions to the PDE is $\dot{V}(u)=r\|u\|_{L_2}^2-\|u_{s}\|_{L_2}^2$. Then, for any $r<\pi^2$, it can be proven that this derivative satisfies $\dot{V}(u)\leq 0$, thus certifying stability of the system~\cite{valmorbida2014semi}. However, verifying that $\dot{V}(u)\leq 0$ for $r\in (0,\pi^2)$ is not trivial, as it necessitates deriving some upper bound on the norm of $u$ in terms of the norm of its derivative $u_{s}$, invoking e.g. the Poincar\'e inequality. In this manner, the representation of the system dynamics as a (polynomial) function of $u$, $u_{s}$, and $u_{ss}$ complicates the task of verifying suitability of even a simple, fixed LF candidate $V(u)=\frac{1}{2}\|u\|^2_{L_2}$.

Similarly, suppose that we adjust the BCs, imposing e.g. a Neumann condition $u_{s}(t,1)=0$. In this case, despite the fact that neither the gradient $\nabla V(u)$ nor the expression for $\dot{u}(t)$ is changed, the derivative $\dot{V}(u)=\nabla V(u) \dot{u}$ no longer satisfies $\dot{V}(u)=r\|u\|^2_{L_2}-\|u_{s}\|_{L_2}^2$ along solutions to the system. More generally, it is unclear how the BCs affect stability properties of the system, and how we can account for them in testing fitness of any LF candidate.

Because of these difficulties associated with verifying suitability of LFs for nonlinear PDEs, most prior work focuses only on limited classes of PDEs with specific BCs, proving results only for the particular system under consideration.
For example, extensive research has been done deriving stability conditions for Navier-Stokes equations~\cite{goulart2012SOS_stability_fluid,huang2015SOS_stability_fluid,ahmadi2019framework,fuentes2022SOS_stability_fluid}, commonly expanding solutions using e.g. a Galerkin basis, and proving decay of a LF using Sum-Of-Squares (SOS) techniques.
Similarly, stability of the Kuramoto-Sivanshinsky Equation was studied in~\cite{goluskin2019Energy_Kuramoto_Sivashinsky}, assuming periodic BCs, and verifying negativity $\dot{V}(u)\leq 0$ of a quadratic LF using discretization.
However, these results apply only to specific PDEs, offering limited insight into how to test stability of other nonlinear systems.

Prior work studying more general systems includes~\cite{fridman2016new}, deriving a Linear Matrix Inequality (LMI) stability test for a class of wave equations $u_{tt}=u_{ss}+f(u,s,t)$, assuming a bound $f(u,s,t)<g$ on the nonlinear term. 
Similarly, stability of classes of 2nd-order, parabolic PDEs is analysed in~\cite{papachristodoulou2006SOS_stability_PDE,valmorbida2015stability,meyer2015stability,mironchenko2019ISS_nonlinear_PDE}, e.g. deriving polynomial positivity conditions for verifying stability of such systems. However, these results too are limited in their applications, and the proposed stability conditions may be challenging to enforce in practice.

Alternatively, the authors of~\cite{korda2022moments,tacchi2022moments,marx2018moments} propose the use of moment methods for analysis of nonlinear PDEs, testing bounds on e.g. energy functionals $V(u)=\|u\|_{L_2}^2$ using Semidefinite Programming (SDP). In order to obtain such bounds, these works study a dual formulation of the PDE system, attempting to use properties of this dual representation to derive properties of the original PDE. However, an intuitive algorithm applying these techniques to test stability of general PDEs has not yet been developed.

In order to construct a comprehensive framework for testing stability of nonlinear PDEs, in this paper, we illustrate how the main obstacle prohibiting the construction of such a framework is the PDE representation itself.
In particular, as noted, the representation of the PDE dynamics in terms of a polynomial function of not only the state $u$ but also its derivatives $(u_{s},u_{ss},\hdots)$ makes it difficult to verify $\dot{V}(u)\leq 0$ for any candidate LF $V(u)$, in general. Moreover, the presence of BCs in the PDE representation further complicates such analysis, as any LF need only satisfy $\dot{V}(u)\leq 0$ for solutions $u$ satisfying the BCs. A more suitable representation of distributed-state systems, then, should ideally satisfy the following two properties:

1. The system dynamics should be represented as a polynomial function of only a single distributed state $v(t)\in L_2$.

2. The representation should not impose any auxiliary constraints on the state $v(t)$, i.e. should be free of BCs and continuity constraints. 

In this paper, we show how, for scalar-valued, quadratic, 1D PDEs, we may derive an equivalent representation satisfying both of these conditions, termed the quadratic Partial Integral Equation (PIE) representation.

\section{Problem Formulation}

\subsection{Notation}\label{sec:notation}

For a given domain $\Omega\subset\R^d$, let $L_2^n[\Omega]$ and $L_{\infty}^{n}[\Omega]$ denote the set of all $\R^n$-valued square-integrable and bounded functions on $\Omega$, respectively, where we omit the domain when clear from context. For $k\in\N$, define Sobolev subspaces $H_{k}^n[\Omega]$ of $L_2^{n}[\Omega]$ for $\Omega\subseteq\R$ as
\begin{align*}
	&H_{k}^{n}[\Omega]=\bl\{\mbf{v}\in L_2^{n}[\Omega] \mid \partial_s^{\alpha}\mbf{v}\!\in\! L_2^n[\Omega],\ \forall \alpha \in\N: \alpha\!\leq k\br\},
\end{align*}
where we write $\partial_{s}^{\alpha}\mbf{v}=\frac{\partial^{\alpha}}{\partial s^{\alpha}}\mbf{v}$, and where we similarly omit the domain when clear from context.

\subsection{Objectives and Approach}

We consider the problem of representation of nonlinear PDEs and parametrization of quadratic (and non-quadratic) forms on a distributed state-space. In particular, suppose we are given a scalar 2nd-order quadratic PDE of the form
\begin{equation}\label{eq:PDE_nonlinear_intro}
	\tbf{PDE:}\qquad u_t(t,s)=c(s)^T {\srbmat{u\\u_s\\u_{ss}\\u^2\\u u_s\\ u_s^2}}(t,s),	\qquad	s\in[a,b],
\end{equation}
with linear boundary conditions of the form
\begin{equation}\label{eq:BCs_intro}
	u(t)\in X_{B}[a,b]:=\bbbl\{\mbf{u}\in H_{2}[a,b]~\bbr\rvert~ B \smallbmat{\mbf{u}(a)\\\mbf{u}_{s}(a)\\\mbf{u}(b)\\\mbf{u}_{s}(b)}=0\bbbr\},
\end{equation}
parameterized by $\{c,B\}\in L_{\infty}^{6}[a,b]\times\R^{2\times 4}$.

The goal of this paper is to create a unified representation of this class of PDEs (extending to higher-order PDEs) and to verify the existence of quadratic Lyapunov functionals which prove stability of such systems. To create such a unified representation and parametrization, in Section~\ref{sec:PDE2PIE}, we will first use the BCs to express the distributed state, $u \in X_B$, in terms of its highest order derivative, $\mbf{v}:= u_{ss} \in L_2$ -- constructing the linear map, $\mcl{T} :\mbf{v} \mapsto u$ so that $u=\mcl{T}\partial_{s}^2 u$ and $\mbf{v}=\partial_s^2 \mcl{T} \mbf{v}$. This mapping is referred to as the Partial Integral (PI) transformation, and has already been studied extensively in the context of linear PDEs~\cite{shivakumar2022GPDE_Arxiv}.

Equipped with this PI transformation, we will show in Section~\ref{sec:tensor_PIs} that any such quadratic PDE (including higher-order PDEs) can be represented as
\begin{align}\label{eq:PIE_quadratic_intro}
	\partial_{t}\mcl{T}\mbf{v}(t)&=\mcl{A}\mbf{v}(t) +\mcl{B}[\mbf{v}(t)\otimes\mbf{v}(t)].
\end{align}
where $[\mbf{v}(t) \otimes \mbf{v}(t)](\theta,\eta)=\mbf{v}(t,\theta) \mbf{v}(t,\eta)$ is the tensor product of $\mbf{v}(t)\in L_2[a,b]$ with itself and defines a distributed monomial basis on $L_2[a,b]$. Although the focus of this paper is on quadratic PDEs, this representation can also be extended to cubic or higher degree PDEs by inclusion of higher degree monomials such as $\mbf{v}(t)\otimes\mbf{v}(t)\otimes\mbf{v}(t)$. The operator $\mcl{A}$, meanwhile, is a PI operator, which for $\mbf{u}\in L_2[a,b]$ takes the form
\begin{equation*}\Resize{\linewidth}{
		(\mcl{A}\mbf{u})(s)=A_{0}(s)\mbf{u}(s)+\!\!\int_{a}^{s}\!\!\! A_1(s,\theta)\mbf{u}(\theta) d\theta	+\!\!\int_{s}^{b}\!\!\! A_2(s,\theta)\mbf{u}(\theta) d\theta.}
\end{equation*}
Accordingly, the operator $\mcl B$ is a tensor product of PI operators, which for $\mbf{w}\in L_2[[a,b]^2]$ takes the form
\begin{align*}
	(\mcl{B}\mbf{w})(s)&=\int_{a}^{s}\!\!\int_{a}^{\theta} Q_1(s,\theta)R_1(s,\eta)\mbf{w}(\theta,\eta) d\eta d\theta	\\ &\quad+\int_{s}^{b}\!\!\int_{a}^{s}Q_2(s,\theta)R_2(s,\eta)\mbf{w}(\theta,\eta)d\eta d\theta \\
	&\qquad +\int_{s}^{b}\!\!\int_{s}^{\theta}Q_3(s,\theta)R_3(s,\eta)\mbf{w}(\theta,\eta)d\eta d\theta.
\end{align*}
Next, in Section~\ref{sec:stability}, we show that for a quadratic Lyapunov functional of the form $V(\mbf{u})=\ip{\mcl{T}\mbf{u}}{\mcl{P}\mcl{T}\mbf{u}}_{L_2}$ (where $\mcl{P}$ is a PI operator decision variable and $\mbf{u}\in L_2[a,b]$) the derivative of this LF can be represented as
\begin{equation*}\Resize{\linewidth}{ \dot{V}(\mbf{u})=\ip{\srbmat{\mbf{u}\\\mbf{u}\otimes\mbf{u}}}{\bmat{\mcl{A}^*\mcl{P}\mcl{T}+\mcl{T}^*\mcl{P}\mcl{A}&\mcl{T}^*\mcl{P}\mcl{B}\\\mcl{B}^*\mcl{P}\mcl{T}&0}\srbmat{\mbf{u}\\\mbf{u}\otimes\mbf{u}}}_{L_2}.}
\end{equation*}
Exploiting the algebraic closure of our class of tensor products of PI operators, we show how to take a quadratic form such as $\ip{\mbf{u}}{\mcl{T}^*\mcl{P}\mcl{B}[\mbf{u}\otimes\mbf{u}]}_{L_2}$ and convert it to a corresponding linear representation so that
\begin{equation*}
	\ip{\mbf{u}}{\mcl{T}^*\mcl{P}\mcl{B}[\mbf{u}\otimes\mbf{u}]}_{L_2}=\mcl{K}[\mbf{u}\otimes\mbf{u}\otimes\mbf{u}].
\end{equation*}
Finally, to test stability of a PDE, in Theorem~\ref{thm:stability_LPI}, we enforce negativity of $\mcl{A}^*\mcl{P}\mcl{T}+\mcl{T}^*\mcl{P}\mcl{A}$ along with the constraint $\mcl{K}=0$, thus ensuring $\dot{V}(\mbf{u})\leq 0$ for all $\mbf{u}\in L_2$. This approach is applied to verify stability of three commonly encountered scalar quadratic PDEs in Section~\ref{sec:examples}.


\section{A Map From Fundamental State to PDE State}\label{sec:PDE2PIE}

We consider an $N$th order, scalar-valued, quadratic PDE of the form
\begin{align}\label{eq:PDE_nonlinear}
	\partial_{t}\mbf{u}(t,s)&=\sum_{i=0}^{N}\alpha_{i}(s)\partial_{s}^{i}\mbf{u}(t,s)	\\ &\quad +\sum_{i=0}^{N-1}\sum_{j=0}^{i}\beta_{ij}(s)\partial_{s}^{i}\mbf{u}(t,s)\partial_{s}^{j}\mbf{u}(t,s),	\quad s\in[a,b],	\notag
\end{align}
with linear boundary conditions of the general form
\begin{equation}\label{eq:BCs}\Scale{0.95}{
		\mbf{u}(t)\in X_{B}[a,b]\!:=\!\bbbl\{\!\mbf{u}\in H_{N}[a,b]~\bbr\rvert B \bmat{\Delta_{s}^{a}\mscr{D}^{N-1}\mbf{u}\\\Delta_{s}^{b}\mscr{D}^{N-1}\mbf{u} }=0\bbbr\},}
\end{equation}
where we define boundary operators $\Delta_{s}^{a}\mbf{u} = \mbf{u}$ and $\Delta_{s}^{b}\mbf{u} = \mbf{u}(b)$ for arbitrary $\mbf{u}\in H_{1}[a,b]$, and where for $\mbf{u}\in H_{N}$ and $k\leq N$ we define $\mscr{D}^{k}\mbf{u}$ as the vector of all derivatives of $\mbf{u}$ up to $k$th order as
\begin{align*}
	\mscr{D}^{k}\mbf{u}:=\bmat{\mbf{u}&\partial_{s} \mbf{u}&\hdots &\partial_{s}^{k} \mbf{u}}^T.
\end{align*}
We define solutions to the PDE as follows.
\begin{defn}[Classical Solution to the PDE]
	For a given initial state $\mbf{u}_{0}\in X_{B}$, we say that $\mbf{u}$ is a classical solution to the quadratic PDE defined by $\{B,[\alpha,\beta]\}$ if $\mbf{u}$ is Frech\'et differentiable, $\mbf{u}(0)=\mbf{u}_{0}$, and for all $t\geq0$, $\mbf{u}(t)\in X_{B}$, and $\mbf{u}(t)$ satisfies~\eqref{eq:PDE_nonlinear}.
\end{defn}

In this section, we show how a suitably well-posed PDE of the form given in~\eqref{eq:PDE_nonlinear} and defined in terms of spatial derivatives of the variable $\mbf{u}(s,t)$ can be equivalently represented in terms spatial integrals of the variable $\partial_{s}^{N}\mbf{u}(s,t)$. The advantage of such a representation will be that it is more compact, defined only using bounded integral operators, and is free of auxiliary constraints in the form of boundary conditions and continuity constraints (i.e. $\partial_{s}^{N}\mbf{u}(t)\in L_2$ whereas $\mbf{u}(t)\in X_{B}$). To begin, we recall the standard definition of a Partial Integral (PI) operator~\cite{peet2019PIE_representation}.

\begin{defn}[Algebras of PI Operators: $\Pi_{3}$,$\Pi_{2}$]
	For a given domain, $[a,b]$, we define the parameter space
	\begin{align*}
		\mcl{N}_{3}:=L_{\infty}[a,b]\times L_{2}[[a,b]^2]\times L_{2}[[a,b]^2],
	\end{align*}
	and its subspace
	\begin{align*}
		\mcl{N}_{2}:=L_{2}[[a,b]^2]\times L_{2}[[a,b]^2].
	\end{align*}
	Then, for any $R:=\{R_{0},R_{1},R_{2}\}\in\mcl{N}_{3}$, we define the associated operator $\mcl{R}:=\mcl{P}_{\{R_0,R_1,R_2\}}$ for $\mbf{u}\in L_2[a,b]$ as
	\begin{equation*}\Resize{\linewidth}{
			(\mcl{R}\mbf{u})(s)=R_{0}(s)\mbf{u}(s) \!+\!\int_{a}^{s}\!\!\!  R_{1}(s,\theta)\mbf{u}(\theta)d\theta \!+\!\int_{s}^{b}\!\!\! R_{2}(s,\theta)\mbf{u}(\theta)d\theta.}
	\end{equation*}
	We say $\mcl{R}\in \Pi_3$ if $\mcl{R} =\mcl P_{\{R_0,R_1,R_2\}}$ for some $\{R_0,R_1,R_2\} \in \mcl N_3$ and $\mcl{R}\in \Pi_2$ if $\mcl{R} =\mcl P_{\{0,R_1,R_2\}}$ for some $\{R_1,R_2\} \in \mcl N_2$. For convenience, we say that $\mcl R$ is a 3-PI operator if $\mcl R \in \Pi_3$ and a 2-PI operator if $\mcl R \in \Pi_2$.
\end{defn}
\smallskip

Defining the 3-PI and 2-PI operators in this manner, it has been shown that both $\Pi_{2}$ and $\Pi_{3}$ (and by extension $\mcl N_2,\mcl N_3$) form *-algebras, being closed under summation, composition, scalar multiplication, and adjoint with respect to $L_2$. We refer to~\cite{shivakumar2022GPDE_Arxiv} for explicit parameter maps defining these operations.

Consider now solutions $\mbf{u}$ of an $N$th order quadratic PDE of the form in~\eqref{eq:PDE_nonlinear}. Since, at any time $t\geq 0$, $\mbf{u}(t)\in X_{B} \subset H_N$ is only $N$th-order spatially differentiable, the $N$th-order spatial derivative of this state need only satisfy $\partial_{s}^{N}\mbf{u}(t)\in L_2[a,b]$ and hence is free of any continuity constraints or boundary conditions. We refer to this derivative $\mbf{v}(t):=\partial_{s}^{N}\mbf{u}(t)$ as the \textit{fundamental state} associated to the PDE, defining a map from the PDE state space $X_{B}$ to the fundamental state space $L_2$ by the differential operator $\partial_{s}^{N}:X_{B}\to L_2$. Then, assuming the boundary conditions to be sufficiently well-posed, we can define an inverse map $\mcl{T}:L_2\to X_{B}$ by a 2-PI operator $\mcl{T}$. In particular, we recall the following result from e.g.~\cite{shivakumar2022GPDE_Arxiv}.
\begin{lem}\label{lem:Tmap}
	Let $X_B$ be as defined in Eqn.~\eqref{eq:BCs} for some $B\in\R^{N\times 2N}$ which satisfies the well-posedness conditions as given in Defn.~9 in~\cite{shivakumar2022GPDE_Arxiv}. If $\mcl{T}\in\Pi_{2}$ is as defined in~\cite{shivakumar2022GPDE_Arxiv}, then for all $\mbf{u}\in X_{B}$ and $\mbf{v}\in L_2$,
	\begin{align*}
		\mbf{u}&=\mcl{T}[\partial_{s}^{N}\mbf{u}],	&	&\text{and}	&
		\mbf{v}&=\partial_{s}^{N}[\mcl{T}\mbf{v}].
	\end{align*}
	More generally, if $\mcl{R}_{j}:=\partial_{s}^{j}\circ\mcl{T}\in\Pi_{2}$ for $j\in\{0,\hdots,N-1\}$, then for all $\mbf{u}\in X_{B}$ and $\mbf{v}\in L_2$,
	\begin{align*}
		\partial_{s}^{j}\mbf{u}&=\mcl{R}_{j}[\partial_{s}^{N}\mbf{u}],	&	&\text{and}	&
		\mcl{R}_{j}\mbf{v}&=\partial_{s}^{j}[\mcl{T}\mbf{v}].
	\end{align*}
\end{lem}
\smallskip
\begin{proof}
	We refer to Thm.~10 and Thm.~12 in\cite{shivakumar2022GPDE_Arxiv} for a proof, as well as for explicit formulae mapping the matrix $B$ to operators $\mcl{T}$ and $\mcl{R}_{j}$.
\end{proof}
\smallskip
Using this lemma, we can express both the PDE state $\mbf{u}(t)$ and its derivatives $\partial_{s}^{j}\mbf{u}(t)$ in terms of the fundamental state $\mbf{v}(t):=\partial_{s}^{N}\mbf{u}(t)$ as $\mbf{u}(t)=\mcl{T}\mbf{v}(t)$ and $\partial_{s}^{j}\mbf{u}(t)=\mcl{R}_{j}\mbf{v}(t)$ for $j\in\{0,\hdots,N-1\}$. Clearly, we can also define an identity operator $\mcl{R}_{N}=I\in\Pi_{3}$, so that $\partial_{s}^{N}\mbf{u}=\mcl{R}_{N}\mbf{v}$. Substituting these relations into the PDE~\eqref{eq:PDE_nonlinear}, we find that if $\mbf{u}$ satisfies the PDE, then $\mbf v$ satisfies
\begin{align}\label{eq:PIE_intermediate}
	\partial_{t}(\mcl{T}\mbf{v})(t,s)&=\sum_{i=0}^{N}\alpha_{i}(s)(\mcl{R}_{i}\mbf{v})(t,s)	\\[-0.3em] &\qquad +\sum_{i=0}^{N-1}\sum_{j=0}^{i}\beta_{ij}(s)(\mcl{R}_{i}\mbf{v})(t,s)(\mcl{R}_{j}\mbf{v})(t,s),		\notag
\end{align}
where $\mbf{v}(t)\in L_2[a,b]$ at each $t\geq 0$ is free of boundary conditions and continuity constraints. Moreover, in this representation, the dynamics are now expressed as a quadratic function of only the fundamental state $\mbf{v}(t)$ and PI operators applied to this state. In the following section, we show how this quadratic function can be expressed more compactly in a linear format $\mcl{C}Z_{2}(\mbf{v})$, by defining a suitable basis of monomials $Z_{2}(\mbf{v})$ in the state $\mbf{v}(t)$, and an associated class of operators $\mcl{C}$ acting on this monomial basis.

\paragraph*{\tbf{Example}}
Suppose that $u(t,s)$ for $s\in[0,1]$ satisfies Burgers' equation with Dirichlet boundary conditions, so that
\begin{align}\label{eq:Burgers_PDE}
	\textbf{PDE:}\qquad &u_{t}(t,s)= \nu u_{ss}(t,s) - u(t,s) u_{s}(t,s),	\\
	\textbf{BCs:}\qquad &u(t,0)=0,\qquad u(t,1)=0.	\notag
\end{align}
For any $\mbf{u}\in L_2[0,1]$, let $\mcl{T},\mcl{R}\in\Pi_{2}$ be defined as
\begin{align}\label{eq:Burgers_Top_Rop}
	\bl(\mcl{T}\mbf{u}\br)(s)&:=\int_{0}^{s}[s-1]\theta\mbf{u}(\theta)d\theta + \int_{s}^{1}s[\theta-1]\mbf{u}(\theta)d\theta,	\notag\\
	\bl(\mcl{R}\mbf{u}\br)(s)&:=\int_{0}^{s}\theta\mbf{u}(\theta)d\theta + \int_{s}^{1}[\theta-1]\mbf{u}(\theta)d\theta.
\end{align}
Now, if we define the fundamental state as $\mbf{v}(t):=u_{ss}(t)\in L_2[0,1]$, then $u(t)=\mcl{T}\mbf{v}(t)$ and $u_{s}(t)=\mcl{R}\mbf{v}(t)$. Furthermore, $\mbf{v}$ satisfies
\begin{align}\label{eq:Burgers_PIE_1}
	\partial_{t}\mcl{T}\mbf{v}(t) 
	=\nu\mbf{v}(t) -(\mcl{T}\mbf{v}(t))(\mcl{R}\mbf{v}(t)).
\end{align}
Conversely, if $\mbf{v}(t)$ satisfies~\eqref{eq:Burgers_PIE_1} and $u(t)=\mcl{T}\mbf{v}(t)$, then $u(t,s)$ satisfies Eqn.~\eqref{eq:Burgers_PDE} for all $s\in[a,b]$.

\section{Polynomial Representation on a Distributed State}

In the previous section, we showed how a quadratic, 2nd-order PDE of form as in~\eqref{eq:PDE_nonlinear_intro}, defined using state $u(t)$, can be equivalently represented in terms of Eqn.~\eqref{eq:PIE_intermediate}, defined using the `fundamental' state $\mbf v(t)=u_{ss}(t)$. However, the expression for the dynamics of $\mbf{v}$ in Eqn.~\eqref{eq:PIE_intermediate} is not suitable for the purpose of computational stability analysis -- being defined using terms such as $(\mcl{T}\mbf{v})(\mcl{R}\mbf{v})$.

To illustrate the difficulty posed by terms such as $(\mcl{T}\mbf{v})(\mcl{R}\mbf{v})$ in testing stability of systems as in~\eqref{eq:PIE_intermediate}, let us briefly recall the problem of computational stability analysis of nonlinear ordinary differential equations (ODEs). When the vector field, $f$, is polynomial of degree $d$, the ODE may be uniquely represented as $\dot x=f(x)=A Z_d(x)$, where $Z_d(x)$ is the vector of monomials in state $x \in \R^n$ of degree $d$ or less. Additionally, positive quadratic Lyapunov functions may be parameterized as $V(x)=\frac{1}{2}x^TPx$, where $P\ge 0$ is a positive matrix variable. The time-derivative of this Lyapunov function is then $\dot V(x)=x^T P A Z_d(x)$ for which we can enforce negativity using the equality constraint
\begin{equation}\label{eq:eq_constraint}
	x^T P A Z_d(x)=-Z_{q}(x)^TQZ_q(x),
\end{equation}
where $Q$ is likewise a positive matrix variable, and $q=\frac{d+1}{2}$. However, enforcing this constraint is made problematic due to the fact that expressions of the form $x^T P A Z_d(x)$ and $Z_q(x)^TQZ_q(x)$ are \textit{not} unique -- i.e. there exist $Q \neq 0$ such that $Z_{q}(x)^T Q Z_q(x)=0$. Therefore, in order to enforce the constraint in~\eqref{eq:eq_constraint}, we must first convert this expression to the linear representation $C Z_{d+1}(x)=0$, which \textit{is} unique.

Returning now to the distributed-state system in~\eqref{eq:PIE_intermediate}, we likewise want to derive a linear representation $\mcl{C}Z_{d}(\mbf{v})$ of terms such as $(\mcl{T}\mbf{v})(\mcl{R}\mbf{v})$. In order to define this linear representation, however, we first have to define a suitable basis of monomials $Z_{d}(\mbf{v})$ on distributed states $\mbf{v}\in L_2$, which we do in the following subsection. In the next subsection, we then define a suitable class of operators $\Pi_{2^{2}}$ and a map $\mcl{T}\times\mcl{R}\mapsto\mcl{C}\in\Pi_{2^2}$ so that $(\mcl{T}\mbf{v})(\mcl{R}\mbf{v})=\mcl{C}Z_{2}(\mbf{v})$, allowing us to express each of the quadratic terms in~\eqref{eq:PIE_intermediate} in a linear format.

\subsection{Monomial Basis on a Distributed State}\label{sec:polynomials}

In order to define a suitable monomial basis for polynomials on a distributed state $\mbf{v}\in L_2[a,b]$, let us first consider what such a basis looks like for a discretized state $v=\bmat{v_1&\hdots&v_{n}}^T=\bmat{\mbf{v}(s_1)&\hdots&\mbf{v}(s_n)}^T\in\R^{n}$. For this vector $v\in\R^{n}$, the basis of monomials $Z_{2}(v)$ of degree 2 in $v$ consists of each of the independent variables $v_{i}$, as well as any product $v_{i}v_{j}$ of these variables for $i,j\in\{1,\hdots,n\}$ -- i.e. the unique terms in the Kronecker product $v \otimes v$. Extending the vector $v\in\R^{n}$ to a function $\mbf{v}\in L_2[a,b]$, we now have that $\mbf{v}(s_i)$ and $\mbf{v}(s_j)$ are independent variables for any $s_i,s_j\in [a,b]$. Thus whereas in finite dimensions, the independent variables are indexed by $i\in\{1,\hdots,n\}$, for functions on $L_2[a,b]$, every point $s\in[a,b]$ defines an independent variable. Accordingly, the degree one monomials in infinite dimensions would simply be the function $\mbf{v}(s)$, where we have made the continuum extension $v_i\mapsto \mbf{v}(s)$. Likewise, for the monomials of degree 2, the products $v_{i}v_{j}$ may be naturally extended to infinite dimensions as $\mbf{v}(s)\mbf{v}(\theta)$, for $s,\theta\in[a,b]$. We may denote these monomials compactly using the tensor product $(\mbf{v}\otimes\mbf{v})(s,\theta)=\mbf{v}(s)\mbf{v}(\theta)$, so that this tensor product $\mbf{v}\otimes \mbf{v} \in L_2[[a,b]^2]$ is the functional equivalent of the homogeneous degree 2 monomial basis.
Generalizing this to arbitrary degrees, the functional equivalent of the homogeneous degree $d\in\N$ monomial basis becomes $\overbrace{\mbf{v}\otimes\cdots\otimes \mbf{v}}^{d \text{ factors}}\in L_{2}[[a,b]^{d}]$, and any distributed polynomial of degree $d$ in $\mbf{v}$ may be expressed in terms of the basis
\begin{equation*}
	Z_{d}(\mbf{v})=\slbmat{\mbf{v}\\\mbf{v}\otimes\mbf{v}\\\vdots\\[-0.6em]\overbrace{\mbf{v}\otimes\cdots\otimes \mbf{v}}^{d \text{ factors}}}.
\end{equation*}
We refer to this basis as the degree $d$ \textit{distributed monomial} basis in $\mbf{v}$. In the following subsection, we show how we may represent quadratic polynomials $(\mcl{T}\mbf{v})(\mcl{R}\mbf{v})$ in a format $\mcl{Q}(\mbf{v}\otimes\mbf{v})$ that is linear in the degree $2$ distributed monomial.

\subsection{Tensor Products of 2-PI Operators}

Having defined a monomial basis for distributed states in terms of tensor products of this state, we now show how we may similarly define a tensor product of 2-PI operators, so that we may express e.g. $(\mcl{T}\mbf{v})(\mcl{R}\mbf{v})$ in the linear format $(\mcl{T}\otimes\mcl{R})(\mbf{v}\otimes\mbf{v})$. To motivate this class of operators, let us again recall the linear representation of scalar polynomials on a finite dimensional state, $x \in \R^n$. In particular, suppose we are given two such polynomials, $t(x)=c_t^T Z_d(x)$ and $r(x)=c_r^T Z_d(x)$, and want to construct a linear representation of $t(x)r(x)$. Such a linear representation is readily computed using the Kronecker (i.e. tensor) product of $c_t$ and $c_r$ as
\[
t(x)r(x)=(c_t \otimes c_r)(Z_d(x)\otimes Z_d(x)).
\]

The goal of this subsection is to define the distributed equivalent of this formula, where now $t$ and $r$ are the distributed polynomials (albeit linear) $\mcl T\mbf v$ and $\mcl R\mbf v$, and where the 2-PI operators $\mcl T$ and $\mcl R$ are the distributed equivalent of the vectors $c_r$ and $c_t$. Specifically, we define a space $\Pi_{2^2}$, consisting of the tensor product of 2-PI operators so that $\mcl{Q}=\mcl T \otimes \mcl R$ for some $\mcl T,\mcl R \in \Pi_2$ implies $\mcl{Q} \in \Pi_{2^2}$.

To illustrate the notion of tensor product for 2-PI operators, consider $\mcl{T}=\mcl{P}_{\{0,T_1,T_2\}}\in\Pi_{2}$ and $\mcl{R}=\mcl{P}_{\{0,R_1,R_2\}}\in\Pi_{2}$. Then we can express the product $(\mcl{T}\mbf{v})(\mcl{R}\mbf{v})$ in terms of $[\mbf{v}\otimes\mbf{v}](\theta,\eta)=\mbf{v}(\theta)\mbf{v}(\eta)$ as\\[-1.6em]
\begin{align}\label{eq:ToR_example}
	&(\mcl{T}\mbf{v})(s)(\mcl{R}\mbf{v})(s)
	=\int_{a}^{s}\!\!\int_{a}^{s}\!\! T_{1}(s,\theta)R_{1}(s,\eta)\mbf{v}(\theta)\mbf{v}(\eta)d\eta d\theta	\notag\\
	&\hspace*{0.1cm}+\int_{s}^{b}\!\!\int_{a}^{s}\!\! T_{2}(s,\theta)R_{1}(s,\eta)\mbf{v}(\theta)\mbf{v}(\eta)d\eta d\theta	\\
	&\hspace*{0.2cm}+\int_{a}^{s}\!\!\int_{s}^{b}\!\! T_{1}(s,\theta)R_{2}(s,\eta)\mbf{v}(\theta)\mbf{v}(\eta)d\eta d\theta	\notag\\
	&\hspace*{0.4cm}+\int_{s}^{b}\!\!\int_{s}^{b}\!\! T_{2}(s,\theta)R_{2}(s,\eta)\mbf{v}(\theta)\mbf{v}(\eta)d\eta d\theta
	=\bl(\mcl{Q}[\mbf{v}\!\otimes\!\mbf{v}]\br)(s),		\notag
\end{align}

\noindent where $\mcl{Q}$ is again a partial integral operator.
Clearly, however, this new PI operator $\mcl{Q} \in \Pi_{2^2}$ (defined on $\mbf{v}\otimes\mbf{v} \in L_2[[a,b]^2]$) has a structure which is distinct from that in $\Pi_2$ and $\Pi_3$.
Instead, $\mcl{Q}$ belongs to a class of ``tensor product PI operators'', which we define as follows.

\begin{defn}[Class of Tensor Product PI Operators, $\Pi_{2^2}$]\label{defn:tensorPI}
	For a given domain, $[a,b]$, we define the parameter space
	\begin{align*}
		\mcl{N}_{2^2}[a,b]=L_{2}[[a,b]^3]\times L_{2}[[a,b]^3]\times L_{2}[[a,b]^3].
	\end{align*}
	Then, for any $Q:=[Q_{1},Q_{2},Q_{3}]\in\mcl{N}_{2^2}$, we define the associated PI operator $\mcl{Q}:=\mcl{P}[Q]$ for $\mbf{w}\in L_2[[a,b]^2]$ as \\[-1.4em]
	{\small
		\begin{align*}
			&(\mcl{Q}\mbf{w})(s)=\int_{a}^{s}\!\!\int_{a}^{\theta}\!\! Q_{1}(s,\theta,\eta)\mbf{w}(\theta,\eta) d\eta d\theta	\\ &+\!\int_{s}^{b}\!\!\int_{a}^{s}\!\!Q_{2}(s,\theta,\eta)\mbf{w}(\theta,\eta)d\eta d\theta \!+\!\int_{s}^{b}\!\!\int_{s}^{\theta}\!\!Q_{3}(s,\theta,\eta)\mbf{w}(\theta,\eta)d\eta d\theta.
	\end{align*}}
	\noindent	Finally, we say that $\mcl{Q}\in \Pi_{2^2}$ if $\mcl{Q} =\mcl P[Q]$ for some $Q:=[Q_{1},Q_{2},Q_{3}]\in\mcl{N}_{2^2}$.
\end{defn}
\smallskip

We remark that the structure in this definition is slightly different from that of the operator $\mcl{Q}$ defining the product $\mcl{Q}[\mbf{v}\otimes\mbf{v}]=(\mcl{T}\mbf{v})(\mcl{R}\mbf{v})$ in~\eqref{eq:ToR_example}. However, we can convert the representation in~\eqref{eq:ToR_example} to one of the form in Defn.~\ref{defn:tensorPI}, defining a map $\mcl{T} \otimes \mcl{R} \rightarrow \mcl{Q} \in \Pi_{2^2}$ such that $(\mcl{T}\mbf{v})(\mcl{R}\mbf{v})=\mcl{Q}[\mbf{v}\otimes\mbf{v}]$.

\begin{defn}[Tensor Product of PI Operators: $\mcl{T} \otimes \mcl{R}$]
	For any $\mcl T=\mcl P_{\{T_{i}\}}\in\Pi_{2}$ and $\mcl R=\mcl P_{\{R_i\}}\in\Pi_{2}$ where $T:=\{T_{1},T_{2}\}\in\mcl{N}_{2}$ and $R:=\{R_{1},R_{2}\}\in\mcl{N}_{2}$, we say that $\mcl{Q}=\mcl{T}\otimes\mcl{R}$ if $\mcl{Q}=\mcl{P}[Q]$, where
	$Q:=[Q_{1},Q_{2},Q_{3}]\in\mcl{N}_{2^2}$ is such that for all $s,\theta,\eta\in[a,b]$, \\[-1.6em]
	\begin{align*}
		Q_{1}(s,\theta,\eta)&:=T_{1}(s,\theta)R_{1}(s,\eta) +T_{1}(s,\eta)R_{1}(s,\theta),	\\
		Q_{2}(s,\theta,\eta)&:=T_{2}(s,\theta)R_{1}(s,\eta) +T_{1}(s,\eta)R_{2}(s,\theta),	\\
		Q_{3}(s,\theta,\eta)&:=T_{2}(s,\theta)R_{2}(s,\eta) +T_{2}(s,\eta)R_{2}(s,\theta).
	\end{align*}
\end{defn}
\smallskip

\begin{prop}\label{prop:tensor_prod_PIs}
	For any $\mcl{T},\mcl{R}\in \Pi_2$ and $\mbf{v} \in L_2[a,b]$, we have
	$
	(\mcl{T}\mbf{v})(\mcl{R}\mbf{v})=\bl((\mcl{T}\otimes\mcl{R})[\mbf{v}\otimes\mbf{v}]\br).
	$
\end{prop}
\begin{proof}
	The proof follows by splliting e.g.
	{\small
		\begin{align*}
			&(\mcl{P}_{\!\{T_{1},0\}}\!\mbf{v})(s)(\mcl{P}_{\!\{R_{1},0\}}\!\mbf{v})(s)
			=\!\int_{a}^{s}\!\!\!\int_{a}^{s}\!\!T_{1}(s,\theta)R_{1}(s,\eta)\mbf{v}(\theta)\mbf{v}(\eta)d\eta d\theta	\\
			&\hspace*{1.0cm} =\int_{a}^{s}\!\!\int_{a}^{\theta}\!\!T_{1}(s,\theta)R_{1}(s,\eta)\mbf{v}(\theta)\mbf{v}(\eta)d\eta d\theta	\\ &\hspace*{2.5cm}+\int_{a}^{s}\!\!\int_{\theta}^{s}\!\!T_{1}(s,\theta)R_{1}(s,\eta)\mbf{v}(\theta)\mbf{v}(\eta)d\eta d\theta,
	\end{align*}}
	and invoking the identity
	\begin{equation}\label{eq:identity_proof}\Scale{0.95}{
			\int_{a}^{s}\!\!\int_{\theta}^{s}\! F(s,\theta,\eta)d\eta d\theta
			=\!\int_{a}^{s}\!\!\int_{a}^{\theta}\! F(s,\eta,\theta)d\eta d\theta.}
	\end{equation}
	A full proof is given in Appendix~\ref{appx:proof_tensor_prod_PIs}.
\end{proof}

We remark that, just as we can use higher degree tensor products to define higher-degree distributed monomials, $(\mbf{v}\otimes\cdots\otimes\mbf{v})(\theta_{1},\hdots,\theta_{d})=\mbf{v}(\theta_{1})\cdots\mbf{v}(\theta_{d})$, we may also define higher degree tensor products of 2-PI operators to map these distributed monomials, so that $(\mcl{R}_{1}\mbf{v})\cdots(\mcl{R}_{d}\mbf{v})=(\mcl{R}_{1}\otimes\cdots\otimes\mcl{R}_{d})(\mbf{v}\otimes\cdots\otimes \mbf{v})$. Again, this extension is left for subsequent publications.

\medskip

\section{A PIE Representation of Quadratic PDEs}\label{sec:tensor_PIs}

We now return to the quadratic representation of the evolution of $\mbf v:=\partial_{s}^{N}\mbf{u}$ given in~\eqref{eq:PIE_intermediate}, of solutions $\mbf{u}$ of the PDE in~\eqref{eq:PDE_nonlinear}. Using the distributed monomials and tensor products of PI operators defined in the previous section, we can now represent each of the quadratic terms in the PDE in terms of the monomials of $\mbf v$ -- i.e. $\mbf{v}\otimes\mbf{v}$. Specifically, the goal of this subsection is to show that the evolution of $\mbf v:=\partial_{s}^{N}\mbf{u}$ is governed by a quadratic PIE of the form
\begin{equation}\label{eq:PIE_quadratic}
	\tbf{PIE:}\qquad \partial_{t}\mcl{T}\mbf{v}(t)=\mcl{A}\mbf{v}(t) +\mcl{B}[\mbf{v}(t)\otimes\mbf{v}(t)],
\end{equation}
where $\mcl A \in \Pi_3$ and $\mcl B \in \Pi_{2^2}$ are given by
\begin{align}\label{eq:Aop_Bop}
	\mcl{A}&=\sum_{i=0}^{N}\mcl M_{\alpha_{i}}\mcl{R}_{i},	&
	\mcl{B}&=\sum_{i=0}^{N-1}\sum_{j=0}^{i}\mcl M_{\beta_{ij}}[\mcl{R}_{i}\otimes\mcl{R}_{j}],
\end{align}
where  $\mcl M_{c}$ denotes the multiplier operator associated to $c\in L_{\infty}[a,b]$, so that $(\mcl{M}_{c}\mbf{v})(s)=c(s)\mbf{v}(s)$.
We note that the fact that $\Pi_{2^2}$ is closed under composition with multiplier operators is relatively clear, but is also stated formally in Prop.~\ref{prop:1dPI_2dPI_comp}.
We define solutions to the quadratic PIE in Eqn.~\eqref{eq:PIE_quadratic} as follows.
\begin{defn}[Classical Solution to the Quadratic PIE]
	For a given initial state $\mbf{v}_{0}\in L_2$, we say that $\mbf{v}$ is a classical solution to the quadratic PIE defined by $\{\mcl{T},[\mcl{A},\mcl{B}]\}$ if $\mbf{v}$ is Frech\'et differentiable, $\mbf{v}(0)=\mbf{v}_{0}$, and for all $t\geq0$, $\mbf{v}(t)$ satisfies~\eqref{eq:PIE_quadratic}.
\end{defn}

The following lemma proves that there exists an invertible map between classical solutions to the PDE~\eqref{eq:PDE_nonlinear}, and classical solutions to the associated PIE~\eqref{eq:PIE_quadratic}.
\begin{lem}\label{lem:PDE2PIE}
	Suppose that that $B\in\R^{N\times 2N}$ satisfies the well-posedness conditions of Defn. 9 in~\cite{shivakumar2022GPDE_Arxiv}, and let the associated operators $\mcl{T},\mcl{R}_{j}\in\Pi_{2}$ for $j\in\{0,\hdots,N-1\}$ be as defined in Lemma~\ref{lem:Tmap}. Let further $\alpha_{k}\in L_{\infty}[a,b]$ for $k\in\{0,\hdots,N\}$, and $\beta_{ij}\in L_{\infty}[a,b]$ for $i\in\{0,\hdots,N-1\}$, $j\in\{0,\hdots,i\}$, and define the operators $\mcl{A}\in\Pi_{3}$ and $\mcl{B}\in\Pi_{2^2}$ as in~\eqref{eq:Aop_Bop}, where $\mcl{R}_{N}=I\in\Pi_{3}$. Then, $\mbf{v}$ is a classical solution to the quadratic PIE defined by $\{\mcl{T},[\mcl{A},\mcl{B}]\}$ with initial state $\mbf{v}_{0}$ if and only if $\mcl{T}\mbf{v}$ is a classical solution to the quadratic PDE defined by $\{B,[\alpha,\beta]\}$ with initial state $\mcl{T}\mbf{v}_{0}$. Conversely, $\mbf{u}$ is a classical solution to the quadratic PDE defined by $\{B,[\alpha,\beta]\}$ with initial state $\mbf{u}_{0}$ if and only if $=\partial_{s}^N\mbf{u}$ is a classical solution to the quadratic PIE defined by $\{\mcl{T},[\mcl{A},\mcl{B}]\}$ with initial state $\partial_{s}^N\mbf{u}_{0}$.
\end{lem}
\begin{proof}
	Let the operators $\mcl{T},\mcl{R}_{j}\in\Pi_{3}$ for $j\in\{0,\hdots,N\}$ be as defined. Then, by Lemma~\ref{lem:Tmap}, for any $\mbf{v}\in L_2[a,b]$,
	\begin{align*}
		\mbf{v}(s)&=\bl(\partial_{s}^{N}[\mcl{T}\mbf{v}]\br)(s),	&	&\text{and}	&
		(\mcl{R}_{j}\mbf{v})(s)&=\bl(\partial_{s}^{j}[\mcl{T}\mbf{v}]\br)(s).
	\end{align*}
	Defining $\mcl{A}\in\Pi_{3}$ as in~\eqref{eq:Aop_Bop}, by linearity of the PI operators, we find then
	\begin{align*}
		(\mcl{A}\mbf{v})(s)
		&=\sum_{i=0}^{N}\mcl (M_{\alpha_{i}}\mcl{R}_{i}\mbf{v})(s)	\\
		&=\sum_{i=0}^{N}\alpha_{i}(s)(\mcl{R}_{i}\mbf{v})(s)	
		=\sum_{i=0}^{N}\alpha_{i}(s)\partial_{s}^{i}(\mcl{T}\mbf{v})(s).
	\end{align*}
	Similarly, defining $\mcl{B}\in\Pi_{2}$ as in~\eqref{eq:Aop_Bop}, by Prop.~\ref{prop:tensor_prod_PIs}, and linearity of the PI operators,
	\begin{align*}
		(\mcl{B}[\mbf{v}\otimes\mbf{v}])(s)
		&=\sum_{i=0}^{N-1}\sum_{j=0}^{i}\mcl (M_{\beta_{ij}}[\mcl{R}_{i}\otimes\mcl{R}_{j}])[\mbf{v}\otimes\mbf{v}])(s)	\\
		&=\sum_{i=0}^{N-1}\sum_{j=0}^{i} \beta_{ij}(s)(\mcl{R}_{i}\mbf{v})(s)(\mcl{R}_{j}\mbf{v})(s)	\\
		&=\sum_{i=0}^{N-1}\sum_{j=0}^{i} \beta_{ij}(s)\partial_{s}^{i}(\mcl{T}\mbf{v})(s)\partial_{s}^{j}(\mcl{T}\mbf{v})(s)
	\end{align*}
	Invoking these relations, it follows that for any $\mbf{v}(t)\in L_2[a,b]$,
	\begin{align*}
		&\partial_{t}(\mcl{T}\mbf{v})(t,s)=(\mcl{A}\mbf{v})(t,s) +(\mcl{B}[\mbf{v}\otimes\mbf{v}])(t,s)	\\
		&\hspace*{5.0cm} \text{ and }~ \mbf{v}(0)=\mbf{v}_{0},	
		\intertext{if and only if}
		&\partial_{t}(\mcl{T}\mbf{v})(t,s)=\sum_{i=0}^{N}\alpha_{i}(s)\partial_{s}^{i}(\mcl{T}\mbf{v})(t,s)\\
		&\hspace*{2.0cm} +\sum_{i=0}^{N-1}\sum_{j=0}^{i} \beta_{ij}(s)\partial_{s}^{i}(\mcl{T}\mbf{v})(t,s)\partial_{s}^{j}(\mcl{T}\mbf{v})(t,s) \\ 
		&\hspace*{5.0cm}\text{ and }~\mcl{T}\mbf{v}(0)=\mcl{T}\mbf{v}_{0}.
	\end{align*}
	By definition, then, $\mbf{v}$ is a classical solution to the quadratic PIE defined by $\{\mcl{T},[\mcl{A},\mcl{B}]\}$ with initial state $\mbf{v}_{0}\in L_2[a,b]$ if and only if $\mcl{T}\mbf{v}$ is a classical solution to the quadratic PDE defined by $\{B,[\alpha,\beta]\}$ with initial state $\mcl{T}\mbf{v}_{0}\in X_{B}[a,b]$.
	
	Conversely, by Lemma~\ref{lem:Tmap}, we also know that for any $\mbf{u}\in X_{B}[a,b]$
	\begin{align*}
		\mbf{u}(s)&=\bl(\mcl{T}[\partial_{s}^{N}\mbf{u}]\br)(s),	&	&\text{and}	&
		\partial_{s}^{j}\mbf{u}(s)&=\bl(\mcl{R}_{j}[\partial_{s}^{N}\mbf{u}]\br)(s),
	\end{align*}
	By linearity of the PI operators, it follows, then, that
	\begin{align*}
		\sum_{i=0}^{N}\alpha_{i}(s)\partial_{s}^{i}\mbf{u}(s)
		&=\sum_{i=0}^{N}\alpha_{i}(s)(\mcl{R}_{i}[\partial_{s}^{N}\mbf{u}])(s)	\\
		&=\sum_{i=0}^{N}\mcl (M_{\alpha_{i}}\mcl{R}_{i}[\partial_{s}^{N}\mbf{u}])(s)
		=(\mcl{A}[\partial_{s}^{N}\mbf{u}])(s).
	\end{align*}
	Similarly, by Prop.~\ref{prop:tensor_prod_PIs}, and linearity of the PI operators,
	\begin{align*}
		&\sum_{i=0}^{N-1}\sum_{j=0}^{i} \beta_{ij}(s)(\partial_{s}^{i}\mbf{u})(s)(\partial_{s}^{j}\mbf{u})(s)	\\
		&=\sum_{i=0}^{N-1}\sum_{j=0}^{i} \beta_{ij}(s)(\mcl{R}_{i}[\partial_{s}^{N}\mbf{u}])(s)(\mcl{R}_{j}[\partial_{s}^{N}\mbf{u}])(s)	\\
		&=\sum_{i=0}^{N-1}\sum_{j=0}^{i}\mcl (M_{\beta_{ij}}[\mcl{R}_{i}\otimes\mcl{R}_{j}][\partial_{s}^{N}\mbf{u}\otimes\partial_{s}^{N}\mbf{u}])(s)	\\
		&=(\mcl{B}[\partial_{s}^{N}\mbf{u}\otimes\partial_{s}^{N}\mbf{u}])(s)
	\end{align*}
	It follows that, for any $\mbf{u}(t)\in X_{B}[a,b]$,
	\begin{equation*}\Resize{\linewidth}{
			\partial_{t}\mbf{u}(t,s)=\!\sum_{i=0}^{N}\alpha_{i}(s)\partial_{s}^{i}\mbf{u}(t,s) +\!\!\sum_{i=0}^{N-1}\sum_{j=0}^{i}\! \beta_{ij}(s)(\partial_{s}^{i}\mbf{u})(t,s)(\partial_{s}^{j}\mbf{u})(t,s)}
	\end{equation*}
	\vspace*{-0.5cm}
	\begin{align*}
		&\hspace*{5.0cm} \text{ and }~ \mbf{u}(0)=\mbf{u}_{0},	
		\intertext{if and only if}
		&\partial_{t}(\mcl{T}[\partial_{s}^{N}\mbf{u}])(t,s)=(\mcl{A}[\partial_{s}^{N}\mbf{u}])(t,s) +(\mcl{B}[\partial_{s}^{N}\mbf{u}\otimes\partial_{s}^{N}\mbf{u}])(t,s)	\\
		&\hspace*{5.0cm}\text{ and }~\partial_{s}^{N}\mbf{u}(0)=\partial_{s}^{N}\mbf{u}_{0}.
	\end{align*}
	By definition, then, $\mbf{u}$ is a classical solution to the quadratic PDE defined by $\{B,[\alpha,\beta]\}$ with initial state $\mbf{u}_{0}\in X_{B}[a,b]$ if and only if $\partial_{s}^{N}\mbf{u}$ is a classical solution to the quadratic PIE defined by $\{\mcl{T},[\mcl{A},\mcl{B}]\}$ with initial state $\partial_{s}^{N}\mbf{u}_{0}\in L_2[a,b]$, concluding the proof.
\end{proof}

\paragraph*{\tbf{Example}} To illustrate the quadratic PIE representation, suppose that $u$ is a solution of Burgers' equation as defined in Eqn.~\eqref{eq:Burgers_PDE}.
Define $\mcl{T},\mcl{R}\in\Pi_{2}$ as in~\eqref{eq:Burgers_Top_Rop}, and define $\mcl{B}\in\Pi_{2^2}$ for $\mbf{w}\in L_2[[0,1]^2]$ as \\[-1.6em]
\begin{align*}
	(\mcl{B}\mbf{w})(s)&=-((\mcl{T}\otimes\mcl{R})\mbf{w})(s)\\
	&=-\int_{0}^{s}\!\!\int_{0}^{\theta}2[s-1]\theta\eta\mbf{w}(\theta,\eta)d\eta d\theta\\
	&\hspace*{0.5cm}-\int_{s}^{1}\!\!\int_{0}^{s}[2s-1][\theta-1]\eta \mbf{w}(\theta,\eta)d\eta d\theta \\
	&\hspace*{1.0cm} -\int_{s}^{1}\!\!\int_{s}^{\theta}2s[\theta-1][\eta-1] \mbf{w}(\theta,\eta)d\eta d\theta.
\end{align*}
Then  $\mbf v:=u_{ss}$ is a solution of the quadratic PIE given by
\begin{equation*}
	\tbf{PIE:}\qquad \partial_{t}\mcl{T}\mbf{v}(t)=\nu\mbf{v}(t) +\mcl{B}[\mbf{v}(t)\otimes\mbf{v}(t)].
\end{equation*}

\smallskip

\section{A Stability Test for Quadratic PDEs}\label{sec:stability}

In the previous sections, we showed that by constructing a basis of distributed monomials on $L_2$ and by representing the dynamics of the PDE using the fundamental state (which lies in $L_2$), we may propose a compact and universal representation of a class of quadratic PDEs using three PI operators: $\mcl{T}$, $\mcl A$ and $\mcl B$, where $\mcl{T},\mcl{A} \in \Pi_3$ are 3-PI operators and $\mcl B \in \Pi_{2^2}$ is a linear combination of tensor products of 2-PI operators. Based on this representation, in this section, we propose a simple stability test using only quadratic storage functions. While such a stability test is necessarily conservative, it is sufficient to verify stability of many common quadratic PDEs and may be later extended to non-quadratic storage functions.

Specifically, we would like to verify the existence of a quadratic Lyapunov functional of the form $V(\mbf{u})=\ip{\mbf{u}}{\mcl{P}\mbf{u}}_{L_2}=\ip{\mcl{T}\mbf{v}}{\mcl{P}\mcl{T}\mbf{v}}_{L_2}$, where $\mcl{P}=\mcl{P}^*\in \Pi_{3}$ is positive with respect to $L_2$. If $\mbf{v}(t)$ satisfies the dynamics of the quadratic PIE in~\eqref{eq:PIE_quadratic}, then \\[-1.8em]

{\small
	\begin{align*}
		&\dot{V}(\mcl T\mbf{v}(t))=\ip{\partial_{t}\mcl{T}\mbf{v}(t)}{\mcl{P}\mcl{T}\mbf{v}(t)}_{L_2} + \ip{\mcl{T}\mbf{v}(t)}{\mcl{P}(\partial_{t}\mcl{T}\mbf{v}(t))}_{L_2}	\\
		&=\!\ip{\srbmat{\mbf{v}(t)\\\mbf{v}(t)\!\otimes\!\mbf{v}(t)}\!}{\!\bmat{\mcl{A}^*\mcl{P}\mcl{T}+\mcl{T}^*\mcl{P}\mcl{A}&\!\mcl{T}^*\mcl{P}\mcl{B}\\\mcl{B}^*\mcl{P}\mcl{T}&\!0}\!\srbmat{\mbf{v}(t)\\\mbf{v}(t)\!\otimes\!\mbf{v}(t)}}_{L_2}.
\end{align*}}
Since it can be shown that $\mbf v$ and $\mbf v \otimes \mbf v$ are linearly independent, we have that $\dot V(\mcl T\mbf v)\le 0$ for all $\mbf{v} \in L_2$ if and only if $\mcl{Q}:=[\mcl{A}^*\mcl{P}\mcl{T}+\mcl{T}^*\mcl{P}\mcl{A}]\preceq 0$ and $\ip{\mbf{v}}{\mcl T^*\mcl P\mcl{B}[\mbf{v}\otimes\mbf{v}]}_{L_2}=0$ for all $\mbf{v}\in L_2$, as stated in the following result.

\begin{prop}\label{prop:cubic_neg}
	Let $\mcl{Q}\in\Pi_{3}$ and $\mcl{R}\in\Pi_{3^2}$, and define $f:L_2\to\R$ as
	\begin{align*}
		f(\mbf{v}):=\ip{\srbmat{\mbf{v}\\\mbf{v}\otimes\mbf{v}}}{\bmat{\mcl{Q}&\mcl{R}\\\mcl{R}^*&0}\srbmat{\mbf{v}\\\mbf{v}\otimes\mbf{v}}}_{L_2}.
	\end{align*}
	Then, for any $\mcl{E}\in\Pi_{3}$, $f(\mbf{v})\leq -\|\mcl{E}\mbf{v}\|_{L_2}^2\leq 0$ for all $\mbf{v}\in L_2$ if and only if $\mcl{Q}\preceq -\mcl{E}^*\mcl{E}$, and $\ip{\mbf{v}}{\mcl{R}[\mbf{v}\otimes\mbf{v}]}_{L_2}=0$ for all $\mbf{v}\in L_2$.
\end{prop}
\begin{proof}
	To prove this result, we first remark that, by definition, $\ip{\mbf{v}}{\mcl{Q}\mbf{v}}_{L_2}\leq -\|\mcl{E}\mbf{v}\|_{L_2}^2=-\ip{\mbf{v}}{\mcl{E}^*\mcl{E}\mbf{v}}_{L_2}$ for all $\mbf{v}\in L_2$ if and only if $\mcl{Q}\preceq -\mcl{E}^*\mcl{E}$. As such, if $\ip{\mbf{v}}{\mcl{R}[\mbf{v}\otimes\mbf{v}]}_{L_2}=0$, it immediately follows that $f(\mbf{v})=\ip{\mbf{v}}{\mcl{Q}\mbf{v}}_{L_2}\leq -\|\mcl{E}\mbf{v}\|_{L_2}^2$ for all $\mbf{v}\in L_2$ if and only if $\mcl{Q}\preceq -\mcl{E}^*\mcl{E}$.
	
	It remains to prove that if $f(\mbf{v})\leq -\|\mcl{E}\mbf{v}\|_{L_2}^2$ for all $\mbf{v}\in L_2$, then $\ip{\mbf{v}}{\mcl{R}[\mbf{v}\otimes\mbf{v}]}_{L_2}= 0$. To prove this implication, suppose that $f(\mbf{v})\leq -\|\mcl{E}\mbf{v}\|_{L_2}^2\leq 0$ for all $\mbf{v}\in L_2$, but assume for contradiction that there exists a function $\mbf{v}^*\in L_2$ such that $\ip{\mbf{v}}{\mcl{R}[\mbf{v}^*\otimes\mbf{v}^*]}_{L_2}\neq 0$. Without loss of generality, we may assume that $\ip{\mbf{v}^*}{\mcl{R}[\mbf{v}^*\otimes\mbf{v}^*]}_{L_2}> 0$, as otherwise we can simply replace $\mbf{v}^*\leftrightarrow -\mbf{v}^*\in L_2$ to obtain the desired inequality. Since $f(\mbf{v})\leq 0$ for all $\mbf{v}\in L_2$, also $f(\mbf{v}^*)\leq 0$, and thus
	\begin{align*}
		\ip{\mbf{v}^*}{\mcl{Q}\mbf{v}^*}_{L_2}&=f(\mbf{v}^*)-2\ip{\mbf{v}^*}{\mcl{R}[\mbf{v}^*\otimes\mbf{v}^*]}_{L_2} <0
	\end{align*}
	Now, define $\lambda:=-\frac{\ip{\mbf{v}^*}{\mcl{Q}\mbf{v}^*}_{L_2}}{\ip{\mbf{v}^*}{\mcl{R}[\mbf{v}^*\otimes\mbf{v}^*]}_{L_2}}>0$
	and let $\hat{\mbf{v}}=\lambda\mbf{v}^*\in L_2$. Then,
	\begin{align*}
		f(\hat{\mbf{v}})&=\ip{\hat{\mbf{v}}}{\mcl{Q}\hat{\mbf{v}}}_{L_2} +2\ip{\hat{\mbf{v}}}{\mcl{R}[\hat{\mbf{v}}\otimes\hat{\mbf{v}}]}_{L_2}	\\
		&=\lambda^2 \ip{\mbf{v}^*}{\mcl{Q}\mbf{v}^*}_{L_2} +2\lambda^3 \ip{\mbf{v}^*}{\mcl{R}[\mbf{v}^*\otimes\mbf{v}^*]}_{L_2}	\\	
		&=\lambda^2 \bbl[\ip{\mbf{v}^*}{\mcl{Q}\mbf{v}^*}_{L_2} +2\lambda \ip{\mbf{v}^*}{\mcl{R}[\mbf{v}^*\otimes\mbf{v}^*]}_{L_2}\bbr]	\\
		&\hspace*{3.5cm}=-\lambda^2 \ip{\mbf{v}^*}{\mcl{Q}\mbf{v}^*}_{L_2}>0,
	\end{align*}
	contradicting the fact that $f(\mbf{v})\leq 0$ for all $\mbf{v}\in L_2$. Hence, for any $\mbf{v}\in L_2$ we must have $\ip{\mbf{v}}{\mcl{R}[\mbf{v}\otimes\mbf{v}]}_{L_2}= 0$. 
\end{proof}

Returning now to our quadratic Lyapunov Functional $V(\mcl{T}\mbf{v})=\ip{\mcl{T}\mbf{v}}{\mcl{P}\mcl{T}\mbf{v}}_{L_2}$, we find that
$\dot V(\mcl T\mbf v)\le 0$ for all $\mbf{v} \in L_2$ if and only if $\mcl{Q}:=[\mcl{A}^*\mcl{P}\mcl{T}+\mcl{T}^*\mcl{P}\mcl{A}]\preceq 0$ and $\ip{\mbf{v}}{\mcl T^*\mcl P\mcl{B}[\mbf{v}\otimes\mbf{v}]}_{L_2}=0$ for all $\mbf{v}\in L_2$. Here, since $\mcl A,\mcl P,\mcl T$ are standard PI operators, the first constraint $\mcl{Q}\preceq 0$ may be enforced using existing functionality of the PIETOOLS software suite~\cite{shivakumar2021PIETOOLS}.
Concentrating, then, on the condition $\ip{\mbf{v}}{\mcl T^*\mcl P\mcl{B}[\mbf{v}\otimes\mbf{v}]}_{L_2}=0$, we first show that the set $\Pi_{2^2}$ of tensor products of PI operators is closed under composition with a 2-PI operator -- so that $\mcl{T}^*,\mcl{P} \in \Pi_3$ and $\mcl{B} \in \Pi_{2^2}$ implies $\mcl{T}^*\mcl{P}\mcl{B} \in \Pi_{2^2}$.
\smallskip

\begin{prop}\label{prop:1dPI_2dPI_comp}
	Let $\mcl{Q}=\mcl{P}_{\{Q_{i}\}}\in\Pi_{3}$ and $\mcl{B}=\mcl{P}[B]\in\Pi_{2^{2}}$ be defined by parameters $Q=\{Q_{0},Q_{1},Q_{2}\}\in\mcl{N}_{3}$ and $B=[B_{1},B_{2},B_{3}]\in\mcl{N}_{2^2}$. Define $\mcl{G}:=\mcl{P}[G]\in\Pi_{2^2}$, where $G:=[G_{1},G_{2},G_{3}]\in\mcl{N}_{2^2}$ is given by
	
	\vspace*{-0.3cm}
	{\small
		\begin{align*}
			&G_{1}(s,\theta,\eta):=\!\int_{a}^{\eta}\!\! R_{13}(s,\zeta,\theta,\eta) d\zeta	
			+\!\int_{\eta}^{\theta}\!\! R_{12}(s,\zeta,\theta,\eta) d\zeta \\
			&+\!\int_{\theta}^{s}\!\! R_{11}(s,\zeta,\theta,\eta) d\zeta +\!\int_{s}^{b}\! R_{21}(s,\zeta,\theta,\eta)  d\zeta +Q_{0}(s)B_{1}(s,\theta,\eta)	,	\\
			&G_{2}(s,\theta,\eta):=\!\int_{a}^{\eta}\!\! R_{13}(s,\zeta,\theta,\eta) d\zeta
			+\!\int_{\eta}^{s}\!\! R_{12}(s,\zeta,\theta,\eta) d\zeta	\\
			&+\!\int_{s}^{\theta}\!\! R_{22}(s,\zeta,\theta,\eta) d\zeta
			+\!\int_{\theta}^{b}\!\! R_{21}(s,\zeta,\theta,\eta) d\zeta +Q_{0}(s)B_{2}(s,\theta,\eta),	\\
			&G_{3}(s,\theta,\eta):=\!\int_{a}^{s}\!\! R_{13}(s,\zeta,\theta,\eta) d\zeta
			+\!\int_{s}^{\eta}\!\! R_{23}(s,\zeta,\theta,\eta) d\zeta	\\
			&+\!\int_{\eta}^{\theta}\!\! R_{22}(s,\zeta,\theta,\eta) d\zeta
			+\!\int_{\theta}^{b}\!\! R_{21}(s,\zeta,\theta,\eta) d\zeta +Q_{0}(s)B_{3}(s,\theta,\eta),
	\end{align*}}
	with $R_{ij}(s,\zeta,\theta,\eta):=Q_{i}(s,\zeta)B_{j}(\zeta,\theta,\eta)$. Then, for any $\mbf{v}\in L_2$,
	\begin{equation*}
		\bl(\mcl{Q}(\mcl{B}\mbf{v})\br)(s)=(\mcl{G}\mbf{v})(s).
	\end{equation*}
\end{prop}
\smallskip
\begin{proof}
	A proof is given in Appendix~\ref{appx:proof_1dPI_2dPI_comp}.
\end{proof}

Applying this result, we can define an operator $\mcl{G}:=\mcl{T}^*\mcl{P}\mcl{B}\in\Pi_{2^2}$, so that we may express $\ip{\mbf{v}}{\mcl{T}^*\mcl{P}\mcl{B}[\mbf{v}\otimes\mbf{v}]}_{L_2}=\ip{\mbf{v}}{\mcl G[\mbf{v}\otimes\mbf{v}]}_{L_2}$.

Now, in order to enforce $\ip{\mbf{v}}{\mcl G[\mbf{v}\otimes\mbf{v}]}_{L_2}\equiv 0$, we note that, in general, the quadratic representation of a polynomial is not unique, and hence it would be conservative to simply enforce the constraint that $\mcl{G}=0$. To resolve this issue, in Prop.~\ref{prop:PImap_ip}, we show that for any $\mcl G \in \Pi_{2^2}$, the quadratic representation of $\ip{\mbf{v}}{\mcl G[\mbf{v}\otimes\mbf{v}]}_{L_2}$ can be converted to a linear representation of the form $\mcl{K}_{\text{lin}}[\mcl G] ( \mbf{v}\otimes\mbf{v}\otimes\mbf{v})=\ip{\mbf{v}}{\mcl G[\mbf{v}\otimes\mbf{v}]}_{L_2}$, where the transformation $\mcl G \rightarrow \mcl K_{\text{lin}}[\mcl G]$ is given in Defn.~\ref{defn:PImap_ip}. Since the linear representations are uniquely defined, we may then enforce the constraint $\ip{\mbf{v}}{\mcl{G}[\mbf{v}\otimes\mbf{v}]}_{L_2}\equiv 0$ without conservatism by requiring $\mcl{K}_{\text{lin}}[\mcl G]=0$.
\begin{defn}\label{defn:PImap_ip}
	For any $\mcl{G}\in\Pi_{2^{2}}$ where $\mcl G=\mcl{P}[G]$ for $G=[G_{1},G_{2},G_{3}]\in\mcl{N}_{2^2}$, we define the operator $\mcl K_{\text{lin}}[\mcl G]:L_2[[a,b]^{3}]\to\R$ as
	\[
	\mcl{K}_{\text{lin}}[\mcl G]\mbf{w}=\int_{a}^{b}\!\!\int_{a}^{s}\!\!\int_{a}^{\theta}K(s,\theta,\eta)\mbf{w}(s,\theta,\eta)\; d\eta d\theta ds,
	\]
	where
	\[
	K(s,\theta,\eta):=G_{1}(s,\theta,\eta) +G_{2}(\theta,s,\eta) +G_{3}(\eta,s,\theta).
	\]
\end{defn}
\smallskip
\begin{prop}\label{prop:PImap_ip}
	For any $\mcl{G}\in\Pi_{2^{2}}$, we have
	\begin{equation*}	\ip{\mbf{v}}{\mcl{G}[\mbf{v}\otimes\mbf{v}]}_{L_2}=\mcl{K}_{\text{lin}}[\mcl G][\mbf{v}\otimes\mbf{v}\otimes\mbf{v}].
	\end{equation*}
	for any $\mbf{v}\in L_2[a,b]$,
\end{prop}
\smallskip
\begin{proof}
	The proof follows by expanding
	{\small
		\begin{align*}
			&\ip{\mbf{v}}{\mcl{G}[\mbf{v}\otimes\mbf{v}]}_{L_2}	
			=\int_{a}^{b}\!\!\int_{a}^{s}\!\!\int_{a}^{\theta}G_{1}(s,\theta,\eta)\mbf{v}(s)\mbf{v}(\theta)\mbf{v}(\eta)\; d\eta d\theta ds	\\
			&\quad +\int_{a}^{b}\!\!\int_{s}^{b}\!\!\int_{a}^{s}G_{2}(s,\theta,\eta)\mbf{v}(s)\mbf{v}(\theta)\mbf{v}(\eta)\; d\eta d\theta ds	\\
			&\qquad +\int_{a}^{b}\!\!\int_{s}^{b}\!\!\int_{s}^{\theta}G_{3}(s,\theta,\eta)\mbf{v}(s)\mbf{v}(\theta)\mbf{v}(\eta)\; d\eta d\theta ds,	
	\end{align*}}
	and invoking e.g. the identity in~\eqref{eq:identity_proof}
	to express all terms using integrals $\int_{a}^{b}\!\int_{a}^{s}\!\int_{a}^{\theta}$.
	A full proof is given in Appendix~\ref{appx:proof_PImap_ip}.
\end{proof}
\smallskip

Applying these result, we can now declare an optimization program for testing stability of a quadratic PDE as follows.

\smallskip
\begin{thm}\label{thm:stability_LPI}
	Let $\{B,[\alpha,\beta]\}$ define a quadratic PDE as in~\eqref{eq:PDE_nonlinear}, and let associated PI operators $\{\mcl{T},[\mcl{A},\mcl{B}]\}$ be as defined in Lem.~\ref{lem:PDE2PIE}. Suppose that there exist $\epsilon,\delta>0$ and $\mcl{P}=\mcl{P}^*\in\Pi_{3}$ such that \\[-2.0em]
	\begin{align}\label{eq:stability_LPI}
		&\mcl{P}\succeq \epsilon I,	\\
		&\mcl{Q}:=[\mcl{A}^*\mcl{P}\mcl{T}+\mcl{T}^*\mcl{P}\mcl{A}]\preceq -\delta\mcl{T}^*\mcl{T},	\notag\\
		&\mcl{K}_{\text{lin}}[\mcl{T}^*\mcl{P}\mcl{B}]=0.	\notag
	\end{align}
	Finally, let $\mu=\|\mcl{P}\|_{\mcl{L}_{L_{2}}}$.
	Then, any solution $\mbf{u}(t)$ to the PDE defined by $\{B,[\alpha,\beta]\}$ satisfies
	\begin{align*}
		\|\mbf{u}(t)\|_{L_2}^2\leq \frac{\mu}{\epsilon}\|\mbf{u}(0)\|_{L_2}^2 e^{-\frac{\delta}{\mu}t}.
	\end{align*}
\end{thm}
\smallskip
\begin{proof}	
	Consider the functional $V:L_2\rightarrow\R$ defined for $\mbf{v}\in L_2$ as
	\begin{align*}
		V(\mbf{v})=\ip{\mcl{T}\mbf{v}}{\mcl{P}\mcl{T}\mbf{v}}_{L_2}\geq \epsilon \|\mcl{T}\mbf{v}\|^2_{L_2}.
	\end{align*}
	Since $\|\mcl{P}\|_{\mcl{L}_{L_2}}=\mu$, this function is bounded from above as
	\begin{align*}
		V(\mbf{v})=
		\ip{\mcl{T}\mbf{v}}{\mcl{P}\mcl{T}\mbf{v}}_{L_2}\leq \mu\|\mcl{T}\mbf{v}\|_{L_2}^2.
	\end{align*}
	Now, let $\mbf{u}$ be an arbitrary solution to the PDE defined by $\{B,[\alpha,\beta]\}$, and fix  $\mbf{v}:=\partial_{s}^{2}\mbf{u}$. Then, by Lemma~\ref{lem:PDE2PIE}, $\mbf{u}=\mcl{T}\mbf{v}$, and $\mbf{v}$ is a solution to the quadratic PIE defined by $\{\mcl{T},[\mcl{A},\mcl{B}]\}$. As such, the temporal derivative of $V$ along $\mbf{v}$ satisfies
	\begin{align*}
		&\dot{V}(\mbf{v}(t))
		=\ip{\partial_{t}\mcl{T}\mbf{v}(t)}{\mcl{P}\mcl{T}\mbf{v}(t)}_{L_2}
		+\ip{\mcl{T}\mbf{v}(t)}{\mcl{P}(\partial_{t}\mcl{T}\mbf{v}(t))}_{L_2}    \\
		&=\!\ip{\bmat{\mcl{A}&\mcl{B}}\!\srbmat{\mbf{v}(t)\\\mbf{v}(t)\!\otimes\!\mbf{v}(t)}}{\mcl{P}\mcl{T}\mbf{v}(t)}_{L_2}	\\ &\hspace*{1.5cm} +\ip{\mcl{T}\mbf{v}(t)}{\mcl{P}\bmat{\mcl{A}&\mcl{B}}\!\srbmat{\mbf{v}(t)\\\mbf{v}(t)\!\otimes\!\mbf{v}(t)}}_{L_2} \\
		&=\!\ip{\srbmat{\mbf{v}(t)\\\mbf{v}(t)\!\otimes\!\mbf{v}(t)}\!}{\!\bmat{\mcl{A}^*\mcl{P}\mcl{T} \!+\! \mcl{T}^*\mcl{P}\mcl{A}&\mcl{T}^*\mcl{P}\mcl{B}\\\mcl{B}^*\mcl{P}\mcl{T}&0}\!\srbmat{\mbf{v}(t)\\\mbf{v}(t)\!\otimes\!\mbf{v}(t)}}_{L_2}    \\
		&=\ip{\mbf{v}(t)}{\mcl{Q}\mbf{v}(t)}_{L_2} +2\ip{\mbf{v}(t)}{\mcl{T}^*\mcl{P}\mcl{B}\thinspace[\mbf{v}(t)\otimes\mbf{v}(t)]}_{L_2}.
	\end{align*}
	Here, since $\mcl{K}_{\text{lin}}[\mcl{T}^*\mcl{P}\mcl{B}]=0$, by Proposition~\ref{prop:PImap_ip} we have
	\begin{equation*}\Resize{\linewidth}
		{\ip{\mbf{v}(t)}{\mcl{T}^*\mcl{P}\mcl{B}[\mbf{v}(t)\!\otimes\!\mbf{v}(t)]}_{L_2}\!\!=\!\mcl{K}_{\text{lin}}[\mcl{T}^*\mcl{P}\mcl{B}][\mbf{v}(t)\!\otimes\!\mbf{v}(t)\!\otimes\!\mbf{v}(t)]\!=\!0.}
	\end{equation*}
	Since also $Q:=[\mcl{A}^*\mcl{P}\mcl{T} \!+\! \mcl{T}^*\mcl{P}\mcl{A}]\preceq-\delta\mcl{T}^*\mcl{T}$, it follows that
	\begin{align*}
		\dot{V}(\mbf{v}(t))\!=\!\ip{\mbf{v}(t)}{\mcl{Q}\mbf{v}(t)}_{L_2}\!\leq\! -\delta\|\mcl{T}\mbf{v}(t)\|^2_{L_2}
		\leq -\frac{\delta}{\mu}V(\mbf{v}(t)).
	\end{align*}
	Applying the Gr\"onwall-Bellman inequality, we find that
	\begin{align*}
		V(\mbf{v}(t))\leq V(\mbf{v}(0))e^{-\frac{\delta}{\mu} t},
	\end{align*}
	and therefore
	\begin{align*}
		\|\mcl{T}\mbf{v}(t)\|^2_{L_2}\leq \frac{\mu}{\epsilon}\|\mcl{T}\mbf{v}(0)\|_{L_2}^2 e^{-\frac{\delta}{\mu} t}.
	\end{align*}
	Finally, since $\mbf{u}=\mcl{T}\mbf{v}$, we conclude that
	\begin{align*}
		\|\mbf{u}(t)\|^2_{L_2}\leq \frac{\mu}{\epsilon}\|\mbf{u}(0)\|_{L_2}^2 e^{-\frac{\delta}{\mu} t}.
	\end{align*}	
\end{proof}

\section{Numerical Examples}\label{sec:examples}

Implementation of the stability test in Theorem~\ref{thm:stability_LPI} requires certain functionality not implemented in PIETOOLS 2022~\cite{shivakumar2021PIETOOLS} (the current release). Specifically, PIETOOLS 2022 does not support PI operators of the class $\Pi_{2^2}$ -- i.e. the operator $\mcl{B}$. As a result, construction of the quadratic PIE representation and use of tensor PI equality constraints requires significant expertise on the part of the user. While we expect such functionality to be included in a future release, for the purposes of this paper, we have created a CodeOcean capsule which allows the user to declare a limited class of quadratic PDEs and then automates the process of construction of the tensor PI representation and Semidefinite Programming (SDP)-based stability test~\cite{CodeOcean_QuadraticPIE}.
An early version of this software was used to produce the results in the following subsections.

\subsection{Burgers' Equation}

Consider Burgers' equation on $s\in[0,1]$, with an added reaction term $ru(t,s)$, and Dirichlet boundary conditions:
\begin{align*}
	\textbf{PDE:}\qquad u_{t}(t,s)&= u_{ss}(t,s) + ru(t,s) - u(t,s) u_{s}(t,s),	\\
	\textbf{BCs:}\qquad u(t,0)&=0,\qquad u(t,1)=0.
\end{align*}
Let $\mbf{v}(t):=u_{ss}(t)$ be the fundamental state, and define PI operators $\mcl{T},\mcl{R}\in\Pi_{2}$ as in~\eqref{eq:Burgers_Top_Rop}
Then, $u(t)=\mcl{T}\mbf{v}(t)$ and $u_{s}(t)=\mcl{R}\mbf{v}(t)$. We obtain an equivalent PIE representation
\begin{align*}
	\textbf{PIE:}\qquad \partial_{t}\mcl{T}\mbf{v}(t) 
	=\bbl[\underbrace{1+r\mcl{T}}_{\mcl A} \quad \underbrace{-(\mcl{T}\otimes\mcl{R})}_{\mcl B}\bbr]\slbmat{\mbf{v}(t)\\\mbf{v}(t)\otimes\mbf{v}(t)}.
\end{align*}

Applying the conditions of Theorem~\ref{thm:stability_LPI} with $\epsilon=10^{-4}$, $\delta=10^{-6}$, the proposed algorithm is able to find a Lyapunov stability proof for any $r\leq 9.8696\approx \pi^2$, which corresponds precisely to the bound obtained in~\cite{peet2019PIE_representation} for the linearization of Burger's equation -- i.e. where we neglect the $u u_s$ term.

Before moving on to the next example, we note that it is well known that the nonlinear term $u u_s$ vanishes when taking the derivative of a candidate Lyapunov functional of the form $V(u)=\|u\|_{L_2}^2=\|\mcl{T}\mbf{v}\|_{L_2}^2$, i.e. letting $\mcl{P}=1$ in~\eqref{eq:stability_LPI}. To illustrate this phenomenon in the tensor PIE representation, note that, for $r=0$, the functional $V(\mbf{v})=\|\mcl{T}\mbf{v}\|_{L_2}^2$ satisfies
\begin{align*}
	\dot{V}(\mbf{v})=\ip{\mbf{v}}{[\mcl{T}^*+\mcl{T}]\mbf{v}}_{L_2} - \ip{\mbf{v}}{\mcl{T}^*[\mcl{T}\otimes\mcl{R}][\mbf{v}\otimes\mbf{v}]}_{L_2}.
\end{align*}
Then, defining $\mcl{K}_{\text{lin}}$ as in Proposition~\ref{prop:PImap_ip}, we find that $\mcl{K}_{\text{lin}}(\mcl{T}^*[\mcl{T}\otimes\mcl{R}])= 0$, and thus $\dot{V}(\mbf{v})=\ip{\mbf{v}}{[\mcl{T}^*+\mcl{T}]\mbf{v}}_{L_2}$. Finally, a simple calculation yields $\mcl{T}^*=\mcl{T}=-\mcl{R}^*\mcl{R}<0$, so that $\dot{V}(\mbf{v})=-2\ip{\mcl{R}\mbf{v}}{\mcl{R}\mbf{v}}_{L_2}=-2\|u_{s}\|^2_{L_2} \le 0$. Of course, one could derive a similar result combining integration by parts with the boundary conditions. However, the advantage of the PIE framework is that such ad hoc manipulations are unnecessary due to the fact that the boundary conditions are embedded in the operators $\mcl T$ and $\mcl R$.

\subsection{Kortweg-De Vries Equation}\label{sec:examples:KdV}

As mentioned in the previous example, the nonlinear term $u u_s$ is well-known to vanish in the derivative of the Lyapunov functional candidate $V(u)=\|u\|_{L_2}^2$ -- as can be proved using integration by parts. For this reason, we now consider a modified version of the Korteweg-De Vries (KdV) equation with an additional, different type of quadratic term -- $u(s)^2$. The equation is defined on $s\in[0,1]$ with Dirichlet-Neumann boundary conditions,
\begin{align*}
	\textbf{PDE:}\quad \!u_{t}(t,s)&= -u_{sss}(t,s) + u(t,s)[r u(t,s)+6u_{s}(t,s)],	\\
	\textbf{BCs:}\quad \!u(t,0)&=0,\qquad u(t,1)=0,\qquad u_{s}(t,1)=0.
\end{align*}
Define the PI operators $\mcl{T},\mcl{R}\in\Pi_{2}$ for $\mbf{v}\in L_2[0,1]$ as
{
	\begin{align*}
		&\bl(\mcl{T}\mbf{v}\br)(s)
		:=\!\int_{0}^{1}\frac{1}{2}[s-1]^2\theta^2 \mbf{v}(\theta)d\theta -\! \int_{s}^{1}\frac{1}{2}[s-\theta]^2\mbf{v}(\theta)d\theta,	\\
		&\bl(\mcl{R}\mbf{v}\br)(s)
		:=\!\int_{0}^{1}[s-1]\theta^2 \mbf{v}(\theta)d\theta -\! \int_{s}^{1} [s-\theta]\mbf{v}(\theta)d\theta.
\end{align*}}
Then, defining fundamental state $\mbf{v}(t):=u_{sss}(t)$ we have $u(t)=\mcl{T}\mbf{v}(t)$ and $u_{s}(t)=\mcl{R}\mbf{v}(t)$. Imposing this relation in the PDE, the system can be equivalently represented as
\begin{equation*}
	\textbf{PIE:}\hspace*{0.3cm} \partial_{t}\mcl{T}\mbf{v}(t)
	=\bbl[\underbrace{-1}_{\mcl A} \quad \underbrace{\mcl{T}\otimes(r\mcl{T}+6\mcl{R})}_{\mcl B}\bbr]\slbmat{\mbf{v}(t)\\\mbf{v}(t)\otimes\mbf{v}(t)}.
\end{equation*}

Applying the conditions of Theorem~\ref{thm:stability_LPI} with $\epsilon=10^{-4}$, $\delta=10^{-6}$, the proposed algorithm is able to find a global Lyapunov stability proof for any $0\leq r\leq 4.6098$.

Although the authors are unaware of any analytic results for stability of the modified KdV equation, the numerical algorithm for each value of $r$ appears to converge to a Lyapunov functional of the form 
\begin{equation*}\Resize{\linewidth}{
		V(u)=C_{0}\!\int_{0}^{1}\!e^{\frac{r}{2}s}u(s)^2 ds +C_{1}\!\int_{0}^{1}\!\!\int_{0}^{1}\!e^{\frac{r}{3}(s+\theta)}u(s)u(\theta)d\theta ds,}
\end{equation*}
for constants $C_{0},C_{1}>0$. Considering the simpler Lyapunov functional candidate $V(u)=\int_{0}^{1}e^{\frac{r}{2}s}u(s)^2 ds$,
it can be shown that, along solutions to the PDE, the temporal derivative of this functional is given by
\begin{align*}
	\dot{V}(u)
	&=\int_{0}^{1}e^{\frac{r}{2}s}\bbl[\frac{r^3}{16}u(s)^2-\frac{3r}{4}u_{s}(s)^2\bbr] ds -\frac{1}{2}u_{s}(0)^2.
\end{align*}
Invoking the Poincar\'e inequality, stability can then be verified for $r\in[0,4.0017]$, though this bound is clearly conservative. See Appendix~\ref{appx:KdV} for more details.

\subsection{Kuramoto-Sivashinsky Equation}\label{sec:examples:KSE}

We now consider the Kuramoto-Sivashinsky Equation (KSE). As in the previous example, we add a quadratic term $ru^2$ to the dynamics, introducing a nonlinear term which does not vanish trivially for quadratic Lyapunov functional candidates. In particular, we consider a system of the form
\begin{align*}
	u_{t}(t,s)\!&=\! -u_{ssss}(t,s) -\!u_{ss}(t,s) -\! u(t,s)[ru(t,s)\!+\!u_{s}(t,s)],\quad		\\
	u(t,0)\!&=\!u(t,1)=u_{s}(t,0)=u_{s}(t,1)=0.
\end{align*}
Define PI operators $\mcl{T},\mcl{R}_{1},\mcl{R}_2\in\Pi_{2}$ for $\mbf{v}\in L_2[0,1]$ as
{\small%
	\begin{align*}
		\bl(\mcl{T}\mbf{v}\br)(s)&:=-\int_{0}^{s}\frac{1}{6}[s-1]^2 \theta^2 [2s\theta-3s+\theta] \mbf{v}(\theta)d\theta\\
		&\qquad\quad -\int_{s}^{1}\frac{1}{6}[\theta-1]^2 s^2[2s\theta-3\theta+s] \mbf{v}(\theta)d\theta,	\\
		\bl(\mcl{R}_{1}\mbf{v}\br)(s)&:=-\int_{0}^{s}\frac{1}{2}[s-1]\theta^2 [2s\theta-3s+1]\mbf{v}(\theta)d\theta \\
		&\qquad\quad -\int_{s}^{1}\frac{1}{2}s[\theta-1]^2[2s\theta +s-2\theta]\mbf{v}(\theta)d\theta,	\\
		\bl(\mcl{R}_{2}\mbf{v}\br)(s)&:=-\int_{0}^{s}\theta^2 [2s\theta-3s-\theta+2]\mbf{v}(\theta)d\theta \\
		&\qquad\quad -\int_{s}^{1}[\theta-1]^2[2s\theta +s-\theta]\mbf{v}(\theta)d\theta.
\end{align*}}
Then, defining fundamental state $\mbf{v}(t):=u_{ssss}(t)$, we have $u(t)=\mcl{T}\mbf{v}(t)$, $u_{s}(t)=\mcl{R}_{1}\mbf{v}(t)$, and $u_{ss}(t)=\mcl{R}_{2}\mbf{v}(t)$. We obtain an equivalent PIE representation as
\begin{align*}
	\partial_{t}\mcl{T}\mbf{v}(t)=\bbl[\underbrace{-\mcl{T}-\mcl{R}_{2}}_{\mcl{A}} \enspace \underbrace{-\mcl{T}\otimes(r\mcl{T}+\mcl{R}_{1})}_{\mcl{B}}\bbr]\slbmat{\mbf{v}(t)\\\mbf{v}(t)\!\otimes\!\mbf{v}(t)}.
\end{align*}
Applying the conditions of Theorem~\ref{thm:stability_LPI} with $\epsilon=10^{-4}$, $\delta=10^{-6}$, the proposed algorithm is able to find a global Lyapunov stability proof for any $r\in[-0.6500,0.7202]$. Again, although we have no analytic stability for this modified KSE, we note that for the Lyapunov functional candidate
\begin{equation*}
	V(u)=\int_{0}^{1}e^{3rs}u(s)^2ds,
\end{equation*}
the temporal derivative along solutions to the PDE satisfies
\begin{equation*}\Resize{\linewidth}{
		\dot{V}(u)
		=\!\int_{0}^{1}\!\!e^{3rs}\bl[-[81r^{4}\!+\!9r^2]u(s)^2 \!+\![36r^2\!+\!2] u_{s}(s)^2\!-\!2u_{ss}(s)^2\br] ds.}
\end{equation*}
Invoking the Poincar\'e inequality, this derivative can be proven to be nonpositive whenever $|r|\leq 0.3608$, though again, this bound is clearly conservative. See Appendix~\ref{appx:KSE} for more details.

\section{Conclusion}

In this paper, we have proposed a new, compact PIE representation of scalar-valued quadratic PDEs, expressed in terms of PI operators on states $\mbf{v}$ and $\mbf{v}\otimes\mbf{v}$. In order to derive this representation, we first defined a new class of PI operators $\Pi_{2^2}$, acting on states $\mbf{v}\otimes\mbf{v}$. We derived formulae for computing the tensor product $\mcl{Q}:=\mcl{T}\otimes\mcl{R}$ of standard PI operators $\mcl{T},\mcl{R}\in\Pi_{2}$, proving that the resulting operator $\mcl{Q}$ belongs to the newly defined class $\Pi_{2^2}$. Using this tensor product, we then derived expressions for operators $\mcl{A}\in\Pi_{3}$ and $\mcl{B}\in\Pi_{2^{2}}$ defining the PIE representation associated to a particular quadratic PDE. Finally, using this PIE representation, we proposed a method for testing existence of a quadratic Lyapunov functional certifying stability of the PDE, posing this test as an optimization problem that can be solved with semidefinite programming. While currently limited to scalar quadratic PDEs, the results of this paper may be extended to higher-degree polynomial PDEs and higher-degree Lyapunov functionals.

\bibliographystyle{IEEEtran}
\bibliography{bibfile}

\clearpage


\onecolumn

\appendix	
	
In this appendix, we prove a number of results regarding tensor products and compositions of partial integral operators. In order to prove these results, we will extensively make use of the following identities for $F\in L_2[[a,b]^2]$:
\begin{align}\label{eq:integral_identities}
		\int_{a}^{s}\!\!\int_{\theta}^{s}\! F(\theta,\eta)d\eta d\theta
		&=\!\int_{a}^{s}\!\!\int_{\eta}^{s}\! F(\eta,\theta)d\theta d\eta
		=\!\int_{a}^{s}\!\!\int_{a}^{\theta}\! F(\eta,\theta)d\eta d\theta,	\\
		\int_{s}^{b}\!\!\int_{\theta}^{b}\! F(\theta,\eta)d\eta d\theta
		&=\!\int_{s}^{b}\!\!\int_{\eta}^{b}\! F(\eta,\theta)d\theta d\eta
		=\!\int_{s}^{b}\!\!\int_{s}^{\theta}\! F(\eta,\theta)d\eta d\theta.	\notag
\end{align}

\subsection{Proof of Proposition~\ref{prop:tensor_prod_PIs}}\label{appx:proof_tensor_prod_PIs}

\begin{prop}\label{prop:tensor_prod_PIs_appx}
	Let $Q:=\{Q_{1},Q_{2}\}\in\mcl{N}_{2}$ and $R:=\{R_{1},R_{2}\}\in\mcl{N}_{2}$ define 2-PI operators $\mcl{Q}=\mcl{P}_{\{0,Q_1,Q_2\}},\mcl{R}=\mcl{P}_{\{0,R_1,R_2\}}\in\Pi_{2}$. Let $B:=[B_{1},B_{2},B_{3}]\in\mcl{N}_{2^2}$, where
	\begin{align*}
		B_{1}(s,\theta,\eta)&:=Q_{1}(s,\theta)R_{1}(s,\eta) +Q_{1}(s,\eta)R_{1}(s,\theta),	\\
		B_{2}(s,\theta,\eta)&:=Q_{2}(s,\theta)R_{1}(s,\eta) +Q_{1}(s,\eta)R_{2}(s,\theta),	\\
		B_{1}(s,\theta,\eta)&:=Q_{2}(s,\theta)R_{2}(s,\eta) +Q_{2}(s,\eta)R_{2}(s,\theta),	
	\end{align*}
	and define $\mcl{B}:=\mcl{P}[B]\in\Pi_{2^2}$. Then, for any $\mbf{v}\in L_2[a,b]$,
	\begin{align*}
		(\mcl{Q}\mbf{v})(s)(\mcl{R}\mbf{v})(s)=(\mcl{B}[\mbf{v}\otimes\mbf{v}])(s).
	\end{align*}
\end{prop}
\begin{proof}
	Let $\mbf{v}\in L_2[a,b]$ be arbitrary. Expanding the product $(\mcl{Q}\mbf{v})(s)(\mcl{R}\mbf{v})(s)$, and applying the identities iN~\eqref{eq:integral_identities}, we find that
	{\small
		\begin{align*}
			&(\mcl{Q}\mbf{v})(s)(\mcl{R}\mbf{v})(s)
			=\bbbl[\int_{a}^{s}Q_{1}(s,\theta)\mbf{v}(\theta)d\theta +\int_{s}^{b}Q_{2}(s,\theta)\mbf{v}(\theta)d\theta \bbbr]	\bbbl[\int_{a}^{s}R_{1}(s,\eta)\mbf{v}(\eta)d\eta +\int_{s}^{b}R_{2}(s,\eta)\mbf{v}(\eta)d\eta\bbbr]	\\
			&=\int_{a}^{s}\!\!\int_{a}^{s}Q_{1}(s,\theta)\mbf{v}(\theta)R_{1}(s,\eta)\mbf{v}(\eta)d\eta d\theta	
			+\int_{s}^{b}\!\!\int_{a}^{s}Q_{2}(s,\theta)\mbf{v}(\theta)R_{1}(s,\eta)\mbf{v}(\eta)d\eta d\theta	\\
			&\qquad+\int_{a}^{s}\!\!\int_{s}^{b}Q_{1}(s,\theta)\mbf{v}(\theta)R_{2}(s,\eta)\mbf{v}(\eta)d\eta d\theta	 +\int_{s}^{b}\!\!\int_{s}^{b}Q_{2}(s,\theta)\mbf{v}(\theta)R_{2}(s,\eta)\mbf{v}(\eta)d\eta d\theta	\\
			&=\int_{a}^{s}\!\!\int_{a}^{\theta}Q_{1}(s,\theta)R_{1}(s,\eta)\mbf{v}(\theta)\mbf{v}(\eta)d\eta d\theta	+\int_{a}^{s}\!\!\int_{\theta}^{s}Q_{1}(s,\theta)R_{1}(s,\eta)\mbf{v}(\theta)\mbf{v}(\eta)d\eta d\theta	
			+\int_{s}^{b}\!\!\int_{a}^{s}Q_{2}(s,\theta)R_{1}(s,\eta)\mbf{v}(\theta)\mbf{v}(\eta)d\eta d\theta	\\
			&\qquad+\int_{s}^{b}\!\!\int_{a}^{s}Q_{1}(s,\eta)R_{2}(s,\theta)\mbf{v}(\theta)\mbf{v}(\eta)d\eta d\theta	 +\int_{s}^{b}\!\!\int_{s}^{\theta}Q_{2}(s,\theta)R_{2}(s,\eta)\mbf{v}(\theta)\mbf{v}(\eta)d\eta d\theta	 \int_{s}^{b}\!\!\int_{\theta}^{b}Q_{2}(s,\theta)R_{2}(s,\eta)\mbf{v}(\theta)\mbf{v}(\eta)d\eta d\theta	\\
			&=\int_{a}^{s}\!\!\int_{a}^{\theta} B_{1}(s,\theta,\eta)[\mbf{v}\otimes\mbf{v}](\theta,\eta) d\eta d\theta	 +\!\int_{s}^{b}\!\!\int_{a}^{s}B_{2}(s,\theta,\eta)[\mbf{v}\otimes\mbf{v}](\theta,\eta)d\eta d\theta +\!\int_{s}^{b}\!\!\int_{s}^{\theta}B_{3}(s,\theta,\eta)[\mbf{v}\otimes\mbf{v}](\theta,\eta)d\eta d\theta
			=(\mcl{B}[\mbf{v}\otimes\mbf{v}])(s).
	\end{align*}}
	We conclude that $(\mcl{Q}\mbf{v})(s)(\mcl{R}\mbf{v})(s)=(\mcl{B}[\mbf{v}\otimes\mbf{v}])(s)$.
\end{proof}

\subsection{Proof of Proposition~\ref{prop:1dPI_2dPI_comp}}\label{appx:proof_1dPI_2dPI_comp}
	
	\begin{prop}\label{prop:1dPI_2dPI_comp_appx}
		Let $\mcl{Q}=\mcl{P}[Q]\in\Pi_{3}$ and $\mcl{B}=\mcl{P}[B]\in\Pi_{2^{2}}$ be defined by parameters $Q=\{Q_{0},Q_{1},Q_{2}\}\in\mcl{N}_{3}$ and $B=[B_{1},B_{2},B_{3}]\in\mcl{N}_{2^2}$. Define
		\begin{align*}
				&G_{1}(s,\theta,\eta):=Q_{0}(s)B_{1}(s,\theta,\eta) +\!\int_{a}^{\eta}\!\! R_{13}(s,\zeta,\theta,\eta) d\zeta	
				+\!\int_{\eta}^{\theta}\!\! Q_{12}(s,\zeta,\theta,\eta) d\zeta +\!\int_{\theta}^{s}\!\! Q_{11}(s,\zeta,\theta,\eta) d\zeta +\!\int_{s}^{b}\! Q_{21}(s,\zeta,\theta,\eta)  d\zeta	,	\\
				&G_{2}(s,\theta,\eta):=Q_{0}(s)B_{2}(s,\theta,\eta) +\!\int_{a}^{\eta}\!\! R_{13}(s,\zeta,\theta,\eta) d\zeta
				+\!\int_{\eta}^{s}\!\! R_{12}(s,\zeta,\theta,\eta) d\zeta	+\!\int_{s}^{\theta}\!\! R_{22}(s,\zeta,\theta,\eta) d\zeta
				+\!\int_{\theta}^{b}\!\! R_{21}(s,\zeta,\theta,\eta) d\zeta,	\\
				&G_{3}(s,\theta,\eta):=Q_{0}(s)B_{3}(s,\theta,\eta) +\!\int_{a}^{s}\!\! R_{13}(s,\zeta,\theta,\eta) d\zeta
				+\!\int_{s}^{\eta}\!\! R_{23}(s,\zeta,\theta,\eta) d\zeta	+\!\int_{\eta}^{\theta}\!\! R_{22}(s,\zeta,\theta,\eta) d\zeta
				+\!\int_{\theta}^{b}\!\! R_{21}(s,\zeta,\theta,\eta) d\zeta,
		\end{align*}
		where $R_{ij}(s,\zeta,\theta,\eta):=Q_{i}(s,\zeta)B_{j}(\zeta,\theta,\eta)$, and let $\mcl{G}:=\mcl{P}[G]$, where $G:=[G_{1},G_{2},G_{3}]\in\mcl{N}_{2^2}$. Then, for any $\mbf{v}\in L_2$,
		\begin{equation*}
			\mcl{Q}(\mcl{B}\mbf{v})(s)=(\mcl{G}\mbf{v})(s).
		\end{equation*}
	\end{prop}
	
\begin{proof}
	In order to prove this result, we note that we can decompose $\mcl{Q}=\mcl{P}_{\{Q_{0},Q_{1},Q_{2}\}}=\mcl{P}_{\{Q_{0},0,0\}}+\mcl{P}_{\{0,Q_{1},0\}}+\mcl{P}_{\{0,0,Q_{2}\}}$. Here, it is easy to see that
	\begin{equation*}\Resize{\textwidth}{
		(\mcl{P}_{\{0,Q_1,0\}}\mcl{B}\mbf{w})(s)
		=		\int_{a}^{s}\!\!\int_{a}^{\theta}\!Q_{0}(s)B_{1}(s,\theta,\eta)\mbf{w}(\theta,\eta)d\eta d\theta +\!\int_{s}^{b}\!\!\int_{a}^{s}\!Q_{0}(s)B_{2}(s,\theta,\eta)\mbf{w}(\theta,\eta)d\eta d\theta +\!\int_{s}^{b}\!\!\int_{s}^{\theta}\!Q_{0}(s)B_{3}(s,\theta,\eta)\mbf{w}(\theta,\eta)d\eta d\theta.}
	\end{equation*}
	In addition, using the identities in~\eqref{eq:integral_identities}, we find that
	{\small
	\begin{align*}
		&(\mcl{P}_{\{0,Q_1,0\}}\mcl{B}\mbf{w})(s)
		=\int_{a}^{s}Q_{1}(s,\zeta)\bbbl[
		\int_{a}^{\zeta}\!\!\int_{a}^{\theta}\!B_{1}(\zeta,\theta,\eta)\mbf{w}(\theta,\eta)d\eta d\theta +\!\int_{\zeta}^{b}\!\!\int_{a}^{\zeta}\!B_{2}(\zeta,\theta,\eta)\mbf{w}(\theta,\eta)d\eta d\theta +\!\int_{\zeta}^{b}\!\!\int_{\zeta}^{\theta}\!B_{3}(\zeta,\theta,\eta)\mbf{w}(\theta,\eta)d\eta d\theta
		\bbbr]d\zeta	\\
		&=\int_{a}^{s}\!\!\int_{a}^{\zeta}\!\!\int_{a}^{\theta}\!\!Q_{1}(s,\zeta)B_{1}(s,\theta,\eta)\mbf{w}(\theta,\eta)d\eta d\theta d\zeta
		+\!\int_{a}^{s}\!\!\int_{\zeta}^{s}\!\!\int_{a}^{\zeta}\!\!Q_{1}(s,\zeta)B_{2}(\zeta,\theta,\eta)\mbf{w}(\theta,\eta)d\eta d\theta d\zeta +\!\int_{a}^{s}\!\!\int_{s}^{b}\!\!\int_{a}^{\zeta}\!\!Q_{1}(s,\zeta)B_{2}(\zeta,\theta,\eta)\mbf{w}(\theta,\eta)d\eta d\theta d\zeta	\\
		&\qquad+\!\int_{a}^{s}\!\!\int_{\zeta}^{s}\!\!\int_{\zeta}^{\theta}\!\!Q_{1}(s,\zeta)B_{3}(\zeta,\theta,\eta)\mbf{w}(\theta,\eta)d\eta d\theta d\zeta +\!\int_{a}^{s}\!\!\int_{s}^{b}\!\!\int_{\zeta}^{\theta}\!\!Q_{1}(s,\zeta)B_{3}(\zeta,\theta,\eta)\mbf{w}(\theta,\eta)d\eta d\theta d\zeta	\\
		&=\int_{a}^{s}\!\!\int_{\theta}^{s}\!\!\int_{a}^{\theta}\!\!Q_{1}(s,\zeta)B_{1}(s,\theta,\eta)\mbf{w}(\theta,\eta)d\eta d\zeta d\theta	
		+\!\int_{a}^{s}\!\!\int_{a}^{\theta}\!\!\int_{a}^{\zeta}\!\!Q_{1}(s,\zeta)B_{2}(\zeta,\theta,\eta)\mbf{w}(\theta,\eta)d\eta d\zeta d\theta +\!\int_{s}^{b}\!\!\int_{a}^{s}\!\!\int_{a}^{\zeta}\!\!Q_{1}(s,\zeta)B_{2}(\zeta,\theta,\eta)\mbf{w}(\theta,\eta)d\eta d\zeta d\theta	\\
		&\qquad+\!\int_{a}^{s}\!\!\int_{a}^{\theta}\!\!\int_{\zeta}^{\theta}\!\!Q_{1}(s,\zeta)B_{3}(\zeta,\theta,\eta)\mbf{w}(\theta,\eta)d\eta d\zeta d\theta +\!\int_{s}^{b}\!\!\int_{a}^{s}\!\!\int_{\zeta}^{\theta}\!\!Q_{1}(s,\zeta)B_{3}(\zeta,\theta,\eta)\mbf{w}(\theta,\eta)d\eta d\zeta d\theta	\\
		&=\int_{a}^{s}\!\!\int_{a}^{\theta}\!\!\int_{\theta}^{s}\!\!Q_{1}(s,\zeta)B_{1}(s,\theta,\eta)\mbf{w}(\theta,\eta)d\zeta d\eta d\theta
		+\!\int_{a}^{s}\!\!\int_{a}^{\theta}\!\!\int_{\eta}^{\theta}\!\!Q_{1}(s,\zeta)B_{2}(\zeta,\theta,\eta)\mbf{w}(\theta,\eta) d\zeta d\eta d\theta +\!\int_{s}^{b}\!\!\int_{a}^{s}\!\!\int_{\eta}^{s}\!\!Q_{1}(s,\zeta)B_{2}(\zeta,\theta,\eta)\mbf{w}(\theta,\eta)d\zeta d\eta d\theta	\\
		&\qquad+\!\int_{a}^{s}\!\!\int_{a}^{\theta}\!\!\int_{a}^{\eta}\!\!Q_{1}(s,\zeta)B_{3}(\zeta,\theta,\eta)\mbf{w}(\theta,\eta)d\zeta d\eta d\theta +\!\int_{s}^{b}\!\!\int_{a}^{s}\!\!\int_{\zeta}^{s}\!\!Q_{1}(s,\zeta)B_{3}(\zeta,\theta,\eta)\mbf{w}(\theta,\eta)d\eta d\zeta d\theta +\!\int_{s}^{b}\!\!\int_{a}^{s}\!\!\int_{s}^{\theta}\!\!Q_{1}(s,\zeta)B_{3}(\zeta,\theta,\eta)\mbf{w}(\theta,\eta)d\eta d\zeta d\theta	\\
		&=\int_{a}^{s}\!\!\int_{a}^{\theta}\bbbl[
		\int_{a}^{\eta}\!\!Q_{1}(s,\zeta)B_{3}(\zeta,\theta,\eta)d\zeta
		+\!\int_{\eta}^{\theta}\!\!Q_{1}(s,\zeta)B_{2}(\zeta,\theta,\eta) d\zeta
		+\!\int_{\theta}^{s}\!\!Q_{1}(s,\zeta)B_{1}(s,\theta,\eta)d\zeta
		\bbbr]\mbf{w}(\theta,\eta)d\eta d\theta	\\
		&\qquad+\!\int_{s}^{b}\!\!\int_{a}^{s}\bbbl[
		 \int_{a}^{\eta}\!\!Q_{1}(s,\zeta)B_{3}(\zeta,\theta,\eta)d\zeta
		 +\!\int_{\eta}^{s}\!\!Q_{1}(s,\zeta)B_{2}(\zeta,\theta,\eta)d\zeta
		 \bbbr]\mbf{w}(\theta,\eta)d\eta d\theta	
		+\!\int_{s}^{b}\!\!\int_{s}^{\theta}\bbbl[\int_{a}^{s}\!\!Q_{1}(s,\zeta)B_{3}(\zeta,\theta,\eta)d\zeta\bbbr]\mbf{w}(\theta,\eta) d\eta d\theta.
	\end{align*}}
	Similarly
	{\small	
	\begin{align*}
		&(\mcl{P}_{\{0,0,Q_1\}}\mcl{B}\mbf{w})(s)
		=\int_{s}^{b}Q_{2}(s,\zeta)\bbbl[
		\int_{a}^{\zeta}\!\!\int_{a}^{\theta}\!B_{1}(\zeta,\theta,\eta)\mbf{w}(\theta,\eta)d\eta d\theta +\!\int_{\zeta}^{b}\!\!\int_{a}^{\zeta}\!B_{2}(\zeta,\theta,\eta)\mbf{w}(\theta,\eta)d\eta d\theta +\!\int_{\zeta}^{b}\!\!\int_{\zeta}^{\theta}\!B_{3}(\zeta,\theta,\eta)\mbf{w}(\theta,\eta)d\eta d\theta
		\bbbr]d\zeta	\\
		&=\int_{s}^{b}\!\!\int_{a}^{s}\!\!\int_{a}^{\theta}\!\!Q_{2}(s,\zeta)B_{1}(s,\theta,\eta)\mbf{w}(\theta,\eta)d\eta d\theta d\zeta +\!\int_{s}^{b}\!\!\int_{s}^{\zeta}\!\!\int_{a}^{\theta}\!\!Q_{2}(s,\zeta)B_{1}(s,\theta,\eta)\mbf{w}(\theta,\eta)d\eta d\theta d\zeta
		+\!\int_{s}^{b}\!\!\int_{\zeta}^{b}\!\!\int_{a}^{\zeta}\!\!Q_{2}(s,\zeta)B_{2}(\zeta,\theta,\eta)\mbf{w}(\theta,\eta)d\eta d\theta d\zeta	\\
		&\qquad+\!\int_{s}^{b}\!\!\int_{\zeta}^{b}\!\!\int_{\zeta}^{\theta}\!\!Q_{2}(s,\zeta)B_{3}(\zeta,\theta,\eta)\mbf{w}(\theta,\eta)d\eta d\theta d\zeta  \\
		&=\int_{a}^{s}\!\!\int_{s}^{b}\!\!\int_{a}^{\theta}\!\!Q_{2}(s,\zeta)B_{1}(s,\theta,\eta)\mbf{w}(\theta,\eta)d\eta d\zeta d\theta +\!\int_{s}^{b}\!\!\int_{\theta}^{b}\!\!\int_{a}^{\theta}\!\!Q_{2}(s,\zeta)B_{1}(s,\theta,\eta)\mbf{w}(\theta,\eta)d\eta d\zeta d\theta
		+\!\int_{s}^{b}\!\!\int_{s}^{\theta}\!\!\int_{a}^{\zeta}\!\!Q_{2}(s,\zeta)B_{2}(\zeta,\theta,\eta)\mbf{w}(\theta,\eta)d\eta d\zeta d\theta	\\
		&\qquad+\!\int_{s}^{b}\!\!\int_{s}^{\theta}\!\!\int_{\zeta}^{\theta}\!\!Q_{2}(s,\zeta)B_{3}(\zeta,\theta,\eta)\mbf{w}(\theta,\eta)d\eta d\zeta d\theta   \\
		&=\int_{a}^{s}\!\!\int_{a}^{\theta}\!\!\int_{s}^{b}\!\!Q_{2}(s,\zeta)B_{1}(s,\theta,\eta)\mbf{w}(\theta,\eta)d\zeta d\eta d\theta +\!\int_{s}^{b}\!\!\int_{a}^{\theta}\!\!\int_{\theta}^{b}\!\!Q_{2}(s,\zeta)B_{1}(s,\theta,\eta)\mbf{w}(\theta,\eta)d\zeta d\eta d\theta
		+\!\int_{s}^{b}\!\!\int_{s}^{\theta}\!\!\int_{a}^{s}\!\!Q_{2}(s,\zeta)B_{2}(\zeta,\theta,\eta)\mbf{w}(\theta,\eta)d\eta d\zeta d\theta	\\
		&\qquad+\!\int_{s}^{b}\!\!\int_{s}^{\theta}\!\!\int_{s}^{\zeta}\!\!Q_{2}(s,\zeta)B_{2}(\zeta,\theta,\eta)\mbf{w}(\theta,\eta)d\eta d\zeta d\theta
		+\!\int_{s}^{b}\!\!\int_{s}^{\theta}\!\!\int_{s}^{\eta}\!\!Q_{2}(s,\zeta)B_{3}(\zeta,\theta,\eta)\mbf{w}(\theta,\eta)d\zeta d\eta d\theta   \\
		&=\int_{a}^{s}\!\!\int_{a}^{\theta}\bbbl[\int_{s}^{b}\!\!Q_{2}(s,\zeta)B_{1}(s,\theta,\eta)d\zeta \bbbr]\mbf{w}(\theta,\eta) d\eta d\theta +\!\int_{s}^{b}\!\!\int_{a}^{s}\bbbl[
		\int_{s}^{\theta}\!\!Q_{2}(s,\zeta)B_{2}(\zeta,\theta,\eta)\mbf{w}(\theta,\eta)d\zeta d\eta d\theta
		+\!\int_{\theta}^{b}\!\!Q_{2}(s,\zeta)B_{1}(s,\theta,\eta)d\zeta
		\bbbr]\mbf{w}(\theta,\eta) d\eta d\theta	\\
		&\qquad+\!\int_{s}^{b}\!\!\int_{s}^{\theta}\bbbl[
		\int_{s}^{\eta}\!\!Q_{2}(s,\zeta)B_{3}(\zeta,\theta,\eta)d\zeta
		+\!\int_{\eta}^{\theta}\!\!Q_{2}(s,\zeta)B_{2}(\zeta,\theta,\eta)d\zeta
		+\!\int_{\theta}^{b}\!\!Q_{2}(s,\zeta)B_{1}(s,\theta,\eta)d\zeta
		\bbbr]\mbf{w}(\theta,\eta) d\eta d\theta
	\end{align*}}
	Combining these results, it is clear that
	\begin{align*}
		(\mcl{Q}\mcl{B}\mbf{w})(s)=(\mcl{P}_{\{Q_{0},0,0\}}\mcl{B}\mbf{w})(s) +(\mcl{P}_{\{0,Q_{1},0\}}\mcl{B}\mbf{w})(s) +(\mcl{P}_{\{0,0,Q_{2}\}}\mcl{B}\mbf{w})(s) =(\mcl{G}\mbf{w})(s)
	\end{align*}
\end{proof}

\subsection{Proof of Proposition~\ref{prop:PImap_ip}}\label{appx:proof_PImap_ip}
\begin{prop}\label{prop:PImap_ip_appx}
	Let PI operator and $\mcl{G}=\mcl{P}[G]\in\Pi_{2^{2}}$ be defined by parameters $G=[G_{1},G_{2},G_{3}]\in\mcl{N}_{2^2}$. Let
	\begin{align*}
		G(s,\theta,\eta):=G_{1}(s,\theta,\eta) +G_{2}(\theta,s,\eta) +G_{3}(\eta,s,\theta),
	\end{align*}
	and define $\mcl{K}:L_2[[a,b]^{3}]\to\R$ as
	\begin{align*}
		\mcl{K}\mbf{w}(s,\theta,\eta)=\int_{a}^{b}\!\!\int_{a}^{s}\!\!\int_{a}^{\theta}K(s,\theta,\eta)\mbf{w}(s,\theta,\eta)d\eta d\theta ds.
	\end{align*}
	Then, for any $\mbf{v}\in L_2[a,b]$, $\ip{\mbf{v}}{\mcl{G}[\mbf{v}\otimes\mbf{v}]}_{L_2}=\mcl{K}[\mbf{v}\otimes\mbf{v}\otimes\mbf{v}]$.
\end{prop}
\begin{proof}
	Applying the identities in~\eqref{eq:integral_identities}, we find that for any $\mbf{v}\in L_2[a,b]$,
	\begin{align*}
		\ip{\mbf{v}}{\mcl{G}[\mbf{v}\otimes\mbf{v}]}_{L_2}
		&=\int_{a}^{b}\mbf{v}(s)(\mcl{G}[\mbf{v}\otimes\mbf{v}])(s) ds	\\
		&=\int_{a}^{b}\!\!\int_{a}^{s}\!\!\int_{a}^{\theta}G_{1}(s,\theta,\eta)\mbf{v}(s)\mbf{v}(\theta)\mbf{v}(\eta)\; d\eta d\theta ds	 +\int_{a}^{b}\!\!\int_{s}^{b}\!\!\int_{a}^{s}G_{2}(s,\theta,\eta)\mbf{v}(s)\mbf{v}(\theta)\mbf{v}(\eta)\; d\eta d\theta ds	\\ &\qquad +\int_{a}^{b}\!\!\int_{s}^{b}\!\!\int_{s}^{\theta}G_{3}(s,\theta,\eta)\mbf{v}(s)\mbf{v}(\theta)\mbf{v}(\eta)\; d\eta d\theta ds	\\
		&=\int_{a}^{b}\!\!\int_{a}^{s}\!\!\int_{a}^{\theta}\bl[G_{1}(s,\theta,\eta)+G_{2}(\theta,s,\eta)\br]\mbf{v}(s)\mbf{v}(\theta)\mbf{v}(\eta)\; d\eta d\theta ds	 +\int_{a}^{b}\!\!\int_{a}^{s}\!\!\int_{\theta}^{s}G_{3}(\theta,s,\eta)\mbf{v}(s)\mbf{v}(\theta)\mbf{v}(\eta)\; d\eta d\theta ds	\\
		&=\int_{a}^{b}\!\!\int_{a}^{s}\!\!\int_{a}^{\theta}[G_{1}(s,\theta,\eta)\!+\! G_{2}(\theta,s,\eta)+G_{3}(\eta,s,\theta)]	\mbf{v}(s)\mbf{v}(\theta)\mbf{v}(\eta)\; d\eta d\theta ds	\\
		&=\int_{a}^{b}\!\!\int_{a}^{s}\!\!\int_{a}^{\theta}K(s,\theta,\eta)[\mbf{v}\otimes\mbf{v}\otimes\mbf{v}](s,\theta,\eta)\; d\eta d\theta ds	\\
		&=\mcl{K}[\mbf{v}\otimes\mbf{v}\otimes\mbf{v}]
	\end{align*}
	We conclude that, for any $\mbf{v}\in L_2[a,b]$, $\ip{\mbf{v}}{\mcl{G}[\mbf{v}\otimes\mbf{v}]}_{L_2}=\mcl{K}[\mbf{v}\otimes\mbf{v}\otimes\mbf{v}]$
\end{proof}

\subsection{Conservative Stability Bound for Modified Korteweg-De Vries Equation}\label{appx:KdV}

Recall the modified Korteweg-De Vries (KdV) equation from Subsection~\ref{sec:examples:KdV},
\begin{align*}
	\textbf{PDE:}\quad u_{t}(t,s)&= -u_{sss}(t,s) + u(t,s)[r u(t,s)+6u_{s}(t,s)],	\\
	\textbf{BCs:}\quad u(t,0)&=0,\qquad u(t,1)=0,\qquad u_{s}(t,1)=0.
\end{align*}
To verify stability of this system, consider the Lyapunov functional candidate
\begin{equation*}
	V(u)=\ip{u}{\mcl{P}_{\{P_{0},0,0\}}u}_{L_2}=\int_{0}^{1}e^{\frac{r}{2}s}u(s)^2 ds,\qquad \text{where}\quad P_{0}(s)=e^{\frac{r}{2}s}.
\end{equation*}
Along solutions to the PDE, this functional satisfies
\begin{equation*}
	\dot{V}(u)=\ip{u_{t}}{\mcl{P}_{\{P_{0},0,0\}}u}_{L_2} +\ip{u}{\mcl{P}_{\{P_{0},0,0\}}u_{t}}_{L_2}	
	=2\ip{u}{\mcl{P}_{\{P_{0},0,0\}}u[ru+6u_{s}]}_{L_2} -2\ip{u}{\mcl{P}_{\{P_{0},0,0\}}u_{sss}}_{L_2}.
\end{equation*}
Here, using integration by parts, and invoking the boundary conditions $u(0)=u(1)=0$, we remark that
\begin{equation*}
	\ip{u}{\mcl{P}_{\{P_{0},0,0\}}u[ru+6u_{s}]}_{L_2}	
	=\int_{0}^{1}e^{\frac{r}{2}s}[6u(s)^2u_{s}(s) +ru(s)^3]ds	
	=\bbl[2e^{\frac{r}{2}s}u(s)^3\bbr]\bbr|_{s=0}^{s} -r\int_{0}^{1}e^{\frac{r}{2}s}\bl[u(s)^3 -u(s)^3\br] ds
	=0.
\end{equation*}
Again using integration by parts, and invoking the boundary conditions $u(0)=u(1)=u_{s}(1)=0$, it also follows that
\begin{align*}
	\dot{V}(u)
	=-2\ip{u}{\mcl{P}_{\{P_{0},0,0\}}u_{sss}}_{L_2}
	&=-2\int_{0}^{1}e^{\frac{r}{2}s}u(s)u_{sss}(s)ds	\\
	&=-2\bbl[e^{\frac{r}{2}s}u(s)u_{ss}(s)\bbr]\bbr|_{s=0}^{1}	 +2\int_{0}^{1}\bbl[\frac{r}{2}e^{\frac{r}{2}s}u(s)+e^{\frac{r}{2}s}u_{s}(s)\bbr]u_{ss}(s)ds	\\
	&=r\bbl[e^{\frac{r}{2}s}u(s)u_{s}(s)\bbr]\bbr|_{s=0}^{1}-\int_{0}^{1}\bbl[\frac{r^2}{2}e^{\frac{r}{2}s}u(s) +re^{\frac{r}{2}s}u_{s}(s)\bbr]u_{s}(s)ds	\\
	&\qquad+\bbl[e^{\frac{r}{2}s}u_{s}(s)^2\bbr]\bbr|_{s=0}^{1}-\int_{0}^{1}\frac{r}{2}e^{\frac{r}{2}s}u_{s}(s)^2ds	\\
	&=-\frac{r^2}{4}\bbl[e^{\frac{r}{2}s}u(s)^2\bbr]\bbr|_{s=0}^{1}
	+\frac{r^3}{8}\int_{0}^{1}e^{\frac{r}{2}s}u(s)^2 ds	
	-r\int_{0}^{1}e^{\frac{r}{2}s}u_{s}(s)^2 ds	\\
	&\qquad-u_{s}(0)^2-\frac{r}{2}\int_{0}^{1}e^{\frac{r}{2}s}u_{s}(s)^2ds \\
	&=\int_{0}^{1}e^{\frac{r}{2}s}\bbl[\frac{r^3}{8}u(s)^2-\frac{3r}{2}u_{s}(s)^2\bbr] ds -u_{s}(0)^2.
\end{align*}
By the Poincar\'e inequality, this will be negative for sufficiently small $r$. In particular, we note that we can bound
\begin{align*}
	\int_{0}^{1}e^{\frac{r}{2}s}\frac{r^3}{8}u(s)^2  ds 
	&\leq e^{\frac{r}{2}}\frac{r^3}{8}\int_{0}^{1}u(s)^2 ds =\frac{r^3 e^{\frac{r}{2}}}{8}\|u\|_{L_2}^2,	&	&\text{and}	&
	-\int_{0}^{1}e^{\frac{r}{2}s}\frac{3r}{2}u_{s}(s)^2 ds
	&\leq -\frac{3r}{2}\int_{0}^{1}u_{s}(s)^2 ds
	=-\frac{3r}{2}\|u_{s}\|_{L_2}^2.
\end{align*}
Here, by the Poincar\'e inequality on the unit interval $s\in[0,1]$, we know that
\begin{equation*}
	\|u\|_{L_2}\leq \frac{1}{\pi}\|u_{s}\|_{L_2},\qquad \forall u\in H_{1}[0,1].
\end{equation*}
Combining these results, it follows that
\begin{equation*}
	\dot{V}(u)
	=\int_{0}^{1}e^{\frac{r}{2}s}\bbl[\frac{r^3}{8}u(s)^2-\frac{3r}{2}u_{s}(s)^2\bbr] ds -u_{s}(0)^2	
	\leq \frac{r^3 e^{\frac{r}{2}}}{8}\|u\|_{L_2}^2 -\frac{3r}{2}\|u_{s}\|_{L_2}^2 -u_{s}(0)^2	
	\leq \bbl[\frac{r^3 e^{\frac{r}{2}}}{8\pi^2} -\frac{3r}{2}\bbr]\|u_{s}\|_{L_2}^2.
\end{equation*}
This expression will be nonpositive whenever $r^2 e^{\frac{r}{2}}\leq 12\pi^2$, which yields a value of roughly $r\approx 4.001777$. Thus, the system is stable whenever $0\leq r\leq 4.00177$, though this bound is likely very conservative.

\subsection{Conservative Stability Bound for Modified Kuramoto-Sivashinsky Equation}\label{appx:KSE}

Recall the Kuramoto-Sivashinsky Equation (KSE) from Subsection~\ref{sec:examples:KSE},
\begin{align*}
	\tbf{PDE:}\quad
	u_{t}(t,s)\!&=\! -u_{ssss}(t,s) -\!u_{ss}(t,s) -\! u(t,s)[ru(t,s)\!+\!u_{s}(t,s)],\quad		\\
	\tbf{BCs:}\quad
	u(t,0)\!&=\!u(t,1)=u_{s}(t,0)=u_{s}(t,1)=0.
\end{align*}
To verify stability of this system, consider the Lyapunov functional candidate
\begin{equation*}
	V(u)=\ip{u}{\mcl{P}_{\{P_{0},0,0\}}u}_{L_2}=\int_{0}^{1}e^{3rs}u(s)^2 ds,\qquad \text{where}\quad P_{0}(s)=e^{3rs}.
\end{equation*}
Along solutions to the PDE, this functional satisfies
\begin{equation*}
	\dot{V}(u)=\ip{u_{t}}{\mcl{P}_{\{P_{0},0,0\}}u}_{L_2} +\ip{u}{\mcl{P}_{\{P_{0},0,0\}}u_{t}}_{L_2}	
	=-2\ip{u}{\mcl{P}_{\{P_{0},0,0\}}u[ru+u_{s}]}_{L_2}	 -2\ip{u}{\mcl{P}_{\{P_{0},0,0\}}[u_{ssss}+u_{ss}]}_{L_2}.
\end{equation*}
Here, using integration by parts, and invoking the boundary conditions $u(0)=u(1)=0$, we remark that that
\begin{equation*}
	\ip{u}{\mcl{P}_{\{P_{0},0,0\}}u[ru+u_{s}]}_{L_2}	
	=\int_{0}^{1}e^{3rs}[u(s)^2u_{s}(s) +ru(s)^3]ds	
	=\bbl[\frac{1}{3}e^{3rs}u(s)^3\bbr]\bbr|_{s=0}^{s} -r\int_{0}^{1}e^{3rs}\bl[u(s)^3 -u(s)^3\br] ds
	=0.
\end{equation*}
Hence,
\begin{equation*}
	\dot{V}(u)
	=-2\ip{u}{\mcl{P}_{\{P_{0},0,0\}}[u_{ssss}+u_{ss}]}_{L_2}
	=-2\ip{u}{\mcl{P}_{\{P_{0},0,0\}}u_{ssss}}_{L_2} -2\ip{u}{\mcl{P}_{\{P_{0},0,0\}}u_{ss}}_{L_2}.
\end{equation*}
Here, again using integration by parts, and invoking the boundary conditions $u(0)=u(1)=u_{s}(0)=u_{s}(1)=0$, it follows that
\begin{align*}
	-\ip{u}{\mcl{P}_{\{P_{0},0,0\}}u_{ssss}}_{L_2}
	&=-\int_{0}^{1}e^{3rs}u(s)u_{ssss}(s)ds	\\
	&=-\bbl[e^{3rs}u(s)u_{sss}(s)\bbr]\bbr|_{s=0}^{1}	
	 +\int_{0}^{1}\bbl[3re^{3rs}u(s)+e^{3rs}u_{s}(s)\bbr]u_{sss}(s)ds	\\
	&=3r\bbl[e^{3rs}u(s)u_{ss}(s)\bbr]\bbr|_{s=0}^{1} 
	-\int_{0}^{1}\bbl[9r^2 e^{3rs}u(s) +3re^{3rs}u_{s}(s)\bbr]u_{ss}(s)ds	\\
	&\qquad+\bbl[e^{3rs}u_{s}(s)u_{ss}(s)\bbr]\bbr|_{s=0}^{1}	
	-\int_{0}^{1}\bbl[3re^{3rs}u_{s}(s) +e^{3rs}u_{ss}(s)\bbr]u_{ss}(s)ds	\\
	&=-9r^2\bbl[e^{3rs}u(s)u_{s}(s)\bbr]\bbr|_{s=0}^{1}	
	+\int_{0}^{1}\bbl[27r^3 e^{3rs}u(s) +9r^2e^{3rs}u_{s}(s)\bbr]u_{s}(s)ds	\\
	&\quad -3r\bbl[e^{3rs}u_{s}(s)^2\bbr]\bbr|_{s=0}^{1} +9r^2\int_{0}^{1}e^{3rs}u_{s}(s)^2 ds	
	 -\int_{0}^{1}e^{3rs}u_{ss}(s)^2 ds	\\
	&=\frac{27r^3}{2}\bbl[e^{3rs}u(s)^2\bbr]\bbr|_{s=0}^{1} -\frac{81r^{4}}{2}\int_{0}^{1}e^{3rs}u(s)^2 ds	\\
	&\quad +9r^2\int_{0}^{1}e^{3rs}u_{s}(s)^2 ds +9r^2\int_{0}^{1}e^{3rs}u_{s}(s)^2 ds	
	 -\int_{0}^{1}e^{3rs}u_{ss}(s)^2 ds	\\
	&=\int_{0}^{1}e^{3rs}\bbl[-\frac{81r^{4}}{2}u(s)^2 +18r^2 u_{s}(s)^2-u_{ss}(s)^2\bbr] ds.
\end{align*}
Similarly, we find that
\begin{align*}
	-\ip{u}{\mcl{P}_{\{P_{0},0,0\}}u_{ss}}_{L_2}
	&=-\int_{0}^{1}e^{3rs}u(s)u_{ss}(s)ds	\\
	&=-\bbl[e^{3rs}u(s)u_{s}(s)\bbr]\bbr|_{s=0}^{1} +\int_{0}^{1}\bbl[3re^{3rs}u(s)+e^{3rs}u_{s}(s)\bbr]u_{s}(s)ds	\\
	&=\frac{3r}{2}\bbl[e^{3rs}u(s)^2\bbr]\bbr|_{s=0}^{1} -\frac{9r^2}{2}\int_{0}^{1}e^{3rs}u(s)^2 ds +\int_{0}^{1}e^{3rs}u_{s}(s)^2 ds	
	=\int_{0}^{1}e^{3rs}\bbl[-\frac{9r^2}{2}u(s)^2+u_{s}(s)^2\bbr] ds.
\end{align*}
Combining these results, it follows that
\begin{align*}
	\dot{V}(u)
	=-2\ip{u}{\mcl{P}_{\{P_{0},0,0\}}u_{ssss}}_{L_2} -2\ip{u}{\mcl{P}_{\{P_{0},0,0\}}u_{ss}}_{L_2}	
	=\int_{0}^{1}e^{3rs}\bbl[-9r^2\bl[9r^{2}+1\br]u(s)^2 +\bl[36r^2+2\br] u_{s}(s)^2-2u_{ss}(s)^2\bbr] ds.
\end{align*}
By the Poincar\'e inequality, this will be negative whenever $|r|$ is sufficiently small. In particular, we note that we can bound
\begin{align*}
	\int_{0}^{1}e^{3rs}\bl[36r^2+2\br]u_{s}(s)^2  ds 
	&\leq e^{3r}\bl[36r^2+2\br]\int_{0}^{1}u_{s}(s)^2 ds =e^{3r}\bl[36r^2+2\br]\|u_{s}\|_{L_2}^2,	\\
	-\int_{0}^{1}e^{\frac{r}{2}s}2u_{ss}(s)^2 ds
	&\leq -2\int_{0}^{1}u_{ss}(s)^2 ds
	=-2\|u_{ss}\|_{L_2}^2.
\end{align*}
Here, by the Poincar\'e inequality on the unit interval $s\in[0,1]$, we know that
\begin{equation*}
	\|u_{s}\|_{L_2}\leq \frac{1}{\pi}\|u_{ss}\|_{L_2},\qquad \forall u\in H_{2}[0,1].
\end{equation*}
Combining these results, it follows that
\begin{align*}
	\dot{V}(u)
	&=\int_{0}^{1}e^{3rs}\bbl[-9r^2\bl[9r^{2}+1\br]u(s)^2 +\bl[36r^2+2\br] u_{s}(s)^2-2u_{ss}(s)^2\bbr] ds	\\
	&\leq -9r^2\bl[9r^{2}+1\br]\int_{0}^{1}e^{3rs}u(s)^2 ds + e^{3r}\bl[36r^2+2\br]\|u_{s}\|_{L_2}^2 -2\|u_{ss}\|_{L_2}^2	\\
	&\leq -9r^2\bl[9r^{2}+1\br]\int_{0}^{1}e^{3rs}u(s)^2 ds +\bbl[\frac{e^{3r}\bl[36r^2+2\br]}{\pi^2} -2\bbr]\|u_{ss}\|_{L_2}^2	
	\qquad\leq \bbl[\frac{e^{3r}\bl[36r^2+2\br]}{\pi^2} -2\bbr]\|u_{ss}\|_{L_2}^2.
\end{align*}
This expression will be nonpositive whenever $e^{3r}[18r^2+1]\leq \pi^2$, which yields a value of roughly $r\approx 0.36082$. Thus, the system is stable whenever $|r|\leq 0.3608$, though this bound is likely very conservative.


\end{document}